\documentclass[smallextended]{svjour3}       
\smartqed  
\usepackage{makecell}
\usepackage{listings}
\usepackage{longtable,multirow}
\usepackage{cite}
\usepackage{algorithm, algorithmic}
\usepackage{amsmath}
\usepackage{amssymb}
\usepackage{graphicx}
\usepackage{float}
\usepackage{booktabs}
\newcommand{\hb}{\hfill $\Box$}
\begin{document}

\title { The sum of root-leaf distance interdiction problem with cardinality constraint by upgrading edges on trees
}

\titlerunning{The sum of root-leaf distance interdiction problem}  

\author{Xiao Li        \and
        Xiucui Guan \and Qiao Zhang \and Xinyi Yin \and Panos M. Pardalos
}


\institute{X. Li \and X.C. Guan \and X.Y. Yin \at School of Mathematics, Southeast University, Nanjing
		210096, Jiangsu, China\\ \email{xcguan@163.com} (Corresponding author: Xiucui Guan)
  \and Q. Zhang   \at 
  Aliyun School of Big Data, Changzhou University, Changzhou 213164, Jiangsu, China
		\and P.M. Pardalos \at
		Center for Applied Optimization,
		Department of Industrial and Systems Engineering, University of Florida, Gainesville, FL, USA		
	}

\date{Received: date / Accepted: date}

\maketitle

\begin{abstract}
A network for the transportation of supplies can be described as a rooted tree with a
 weight of a degree of congestion for each edge. We take the sum of root-leaf distance (SRD) on a rooted tree as the whole degree of congestion of the tree. Hence, we 
 consider the SRD interdiction problem on trees with cardinality constraint by upgrading edges  (denoted by (\textbf{SDIPTC}) in brief). 
 It aims to maximize  the SRD by upgrading 
 the weights of $N$ critical edges such that the total upgrade cost
 under some measurement is upper-bounded by a given value. 
 The relevant minimum cost problem(\textbf{MCSDIPTC}) aims to minimize
 the total upgrade cost on the premise that the SRD is lower-bounded by a given value.  
 We develop two different norms including weighted $l_\infty$ norm and weighted bottleneck Hamming distance to measure the upgrade cost. 
 We propose two binary search algorithms within O($n\log n$) time for the problems (\textbf{SDIPTC}) under the two norms, respectively. 
 For problems (\textbf{MCSDIPTC}),we propose two binary search algorithms within O($N n^2$) and O($n \log n$) 
 under weighted $l_\infty$ norm and weighted bottleneck Hamming distance, respectively. These problems are solved through their subproblems (\textbf{SDIPT}) and (\textbf{MCSDIPT}), in which we ignore the cardinality constraint on the number of upgraded edges.
 Finally, we design numerical experiments to show the effectiveness of these algorithms.
\keywords{Network interdiction problems \and Tree \and Greedy algorithm \and Binary search method \and $l_\infty$ norm \and Bottleneck Hamming distance}
\subclass{90C27 \and 90C35 }
\end{abstract}

\section{Introduction}
\label{intro}
Terrorism is a persistent challenge for many nations worldwide, and it is imperative that governments take appropriate measures to prevent and mitigate the impact of terrorist attacks. Among the diverse tactics used by terrorists, actions that trigger fires and hinder the aid process are recognized as Network Interdiction Problems (NIPs), which are based on game theory.

An NIP involves two main actors, a leader and a follower, who have competing objectives. The follower aims to optimize their objectives by using a network to facilitate the movement of resources such as supply convoys or maximizing the amount of material transported through the network. Conversely, the leader seeks to impede the follower's objectives by disrupting the network's structure through interdicting arcs or nodes. The leader can do so by attacking nodes to reduce their capacity or destroy them, ultimately 
hindering the follower's ability to move through the network \cite{ref1}.

Some transportation networks can be represented as rooted trees with weighted edges, where the root node is the critical infrastructure or facility that the leader intends to attack, the internal nodes are served as transportation hubs, and leaf nodes are emergency agencies including fire stations, police stations, hospitals etc.  The edge weight indicates the degree of congestion of the roads, and targeting these edges can disrupt the network by impeding rescue. However, decision-makers with limited resources must identify the most strategic edges to attack in order to maximize the network's overall performance degradation.

In a broader sense, tree networks can be used to describe networks with one-directional links. For instance, during wartime, ammunition supplies may only flow from logistics centers to the front lines. The related issues can be formulated as follows.

Let $T=(V,E,w,u,c)$ be an edge-weighted tree rooted at $s$,
where $V:=\left \{ s,v_1,v_2,\dots,v_n \right \} $ and
$E:=\left \{ e_1,e_2,\dots,e_n \right \}$
are the sets of nodes and edges, respectively.
Let $Y:=\left \{ t_1,t_2,\dots,t_r \right \}$ be the set of leaf nodes.
Let $w(e)$ and $u(e)$ be the original weight and upper bound of 
edge $e\in E$, respectively,
where $u(e)\geq w(e) \geq 0$.
Let $c(e)>0$ be the unit modification cost of edge $e\in E$. Let $P_k:=P_{t_k}:=P_{s,t_k}$ be the root-leaf path from the root node $s$ to the leaf node $t_k$.
Let $w(P_k):=\sum_{e\in P_{s,t_k}}w(e)$ and $w(T):=\sum_{t_k\in Y}w(P_k)$ be the weight of the path $P_k$ and the sum of root-leaf distance (SRD)
 under the edge weight $w$, respectively.
The sum of root-leaf distance  interdiction problem on trees with 
cardinality constraint (\textbf{SDIPTC}) 
aims to maximize the SRD by determining an edge weight vector $\tilde{w}$ such that the modification cost $\|\tilde{w}-w\|$ in a certain norm and the number of updated edges do not exceed two given values $K$ and $N$, respectively.
Its mathematical model can be stated as follows.
\begin{eqnarray}
	&\max \limits_{\tilde{w}}& \tilde{w}(T):=\sum_{t_k \in Y}\tilde{w}(P_k)\nonumber\\
	\textbf { (SDIPTC)} & s.t. &\ \| \tilde{w}-w\|  \leq K,\nonumber\\ 
	&&\sum_{e\in E}H(\tilde{w}(e),w(e))\le N, \label{SDIPTC}\\ 
	&& w(e) \leq \tilde{w}(e) \leq u(e), \quad e \in E,\nonumber
\end{eqnarray}
where $H(\tilde{w}(e),w(e))=\begin{cases}
	0, &{\tilde{w}}(e)=w(e)\\
	1,&{\tilde{w}}(e)\ne w(e)\\
\end{cases}\nonumber 
$ is the Hamming distance between 
$\tilde{w}(e)$ and $w(e)$.

We also consider its relevant minimum cost problem with cardinality constraint (\textbf{MCSDIPTC}) by exchanging its objective function and the norm constraint. The mathematical model of the problem
is as follows.
\begin{eqnarray}
	&\min \limits_{\tilde{w}} & C(\tilde{w}):=\ \| \tilde{w}-w\| \nonumber  \\
	\textbf {(MCSDIPTC) } &s.t.&\sum_{t_k \in Y} \tilde{w}(P_k) \geq D, \nonumber\\
	&&\sum_{e\in E}H(\tilde{w}(e),w(e))\le N,\label{MCSDIPTC} \\
	&&w(e) \leq \tilde{w}(e) \leq u(e), \quad e \in E\nonumber.
\end{eqnarray}

Most network interdiction problems aim to make the network's performance as poor as possible by deleting critical edges or nodes.
Albert (2000) \cite{ref2} found that in a scale-free network when the most connected nodes are removed,
the diameter of the network rapidly increases. They pointed out that this vulnerability that is susceptible to attack is due to the uneven distribution of
connections and the removal of a limited number of nodes in the network can have a significant impact, giving rise to 
the concept of critical nodes and critical edges.
Subsequently, Ayyildiz et al. (2019) \cite{ref3} analyzed how to maximize the length of the shortest path between two
nodes based on a limited interdicting budget from two opposing aspects (breaking and maintaining the network).
On the other hand, Magnouche and Martin (2020) \cite{ref4} studied how to delete the minimum number of critical nodes 
to make the root-leaf path length in the remaining graph at least $d$. They proved the  $\mathcal{NP}$-hardness of the problem, proposed
an integer linear programming model with an exponent $p$ constraints and designed a branch and bound algorithm.

Corley and Sha (1982) \cite{ref5} first applied the NIP by  deleting $K$ critical edges to maximize the length of the 
shortest path. 
Ball et al. (1989) \cite{ref6} proved that the problem is  $\mathcal{NP}$-hard.
Khachiyan et al. (2008) \cite{ref7} proved that there is no approximation algorithm with a performance ratio of 2.
Bazgan et al. (2015) \cite{ref8} provided an $O(mn)$ time algorithm when the increment $b=1$ of the path length, and further (2019) \cite{ref9} proved that the problem is  $\mathcal{NP}$-hard when $b \ge 2$. 

However, it is often difficult to implement the deletion of critical 
edges/nodes in practical applications.
Zhang et al. (2021) \cite{ref10} proposed a concept of updating 
critical edges to replace the limitations of the traditional 
NIP caused by deleting critical edges,
and studied the shortest path interdiction problem (SPIT) 
and its minimum cost problem (MCSPIT) on trees. For the problem (SPIT), they provided an $O(n^2)$ primal-dual algorithm under the weighted $l_1$ norm and improved the time complexity to $O(n)$ under the unit $l_1$ norm.
For the problem (SPIT$_{uH}$) under unit Hamming distance, they \cite{ref11} devised algorithms with time complexities $O(N + l \log l)$ and $O(n(\log n + K^3))$ when $K = 1$
and  $K > 1$, respectively, where $n$, $K$
and $l$ are the numbers of tree nodes, upgraded edges and leaves respectively. 
Afterward, Yi et al. (2023) \cite{ref14} improved the time complexities to $O(n)$ and $O(nK^2)$, respectively.
For the relevant 
minimum cost problem (MCSPIT$_{uH}$), based on dynamic programming 
ideas and binary search methods, Zhang et al. (2020)
\cite{ref11} provided an $O(n^4 \log n)$ algorithm. Yi et al. (2023)\cite{ref14} 
improved the complexity to $O(n^3 \log n)$ through a binary search.
Zhang et al. (2020) \cite{ref11} subsequently demonstrated that the problems 
(SPIT$_H$) and (MCSPIT$_H$) under the weighted Hamming distance are equivalent to the 0-1 knapsack problem and therefore  $\mathcal{NP}$-hard.
Furthermore, for the SRD interdiction problems (SDIPT-UE/N) and (MCSDIPT-UE/N) by upgrading edges/nodes under weighted Hamming distance, Zhang et al. \cite{ref12}
proved its  $\mathcal{NP}$-hardness by showing
the equivalence between the two problems and the 0–1 knapsack problem. 
Under weighted $l_1$ norm, the problems (SDIPT$_1$-UE) and (MCSDIPT$_1$-UE) can be solved in $O(n)$ time by continuous knapsack problems.
For the problem (SDIPT$_{uH}$-UE) under unit Hamming distance  and 
unit node cost problem (SDIPT-UN),
they designed two linear time greedy algorithms. Moreover, 
for the minimum cost problems (MCSDIPT$_{uH}$-UE)under unit Hamming distance 
and unit node cost problem (MCSDIPT-UN), 
they \cite{ref12} used a binary search method to provide algorithms 
with a time complexity of $O(n \log n)$.

Based on the above analysis, we can know that there is no reseach results for the SRD interdiction problems under $l_\infty$ norm and bottleneck Hamming distance, which are sub-problems of the problems (\textbf{SDIPTC}) and (\textbf{MCSDIPTC}) we mainly studied in this paper. We impose a cardinality constrain on the number of modified edges since there are always many optimal solutions, which include some unnecessary modifications of weights, for the problem (\textbf{SDIPT}) under bottleneck norms.
We compare our results with existing studies in Table \ref{table_compare},
where the subscripts $1, u1, H, uH, BH$ and $_\infty$ denote weighted $l_1$-norm, unit $l_1$-norm,
weighted Hamming distance, unit Hamming distance, bottleneck Hamming distance, and weighted $l_\infty$ norm, respectively.

The paper is structured as follows. In Section \ref{sec2}, we propose two algorithms for solving problems (\textbf{SDIPTC$_\infty$}) and (\textbf{MCSDIPTC$_\infty$}) in $O(n)$ and O($N n^2$) time, respectively, which mainly call the subroutines for solving the two sub-problems (\textbf{SDIPT$_ \infty$}) and (\textbf{MCSDIPT$_ \infty$}) in $O(n)$  and O($n \log n $) time, respectively. 
In Section \ref{sec3}, we solve the problems (\textbf{MCSDIPTC$_{BH}$}) and (\textbf{SDIPTC$_{BH}$}) under bottleneck Hamming distance in $O(n)$  and O($n \log n$) time, respectively, which have the same time complexities for the sub-problems without the cardinality constraints. 
In Section \ref{sec4}, we report on computational experiments that demonstrate the effectiveness of all proposed polynomial time algorithms.
Finally, in Section \ref{sec5}, we summarize our findings and suggest some areas for future research.

\begin{table}
	\centering
	\caption{The relationship between the previous research and our research} \label{table_compare}
	\resizebox{1.0\textwidth}{!}
	{
		\begin{tabular}{|c|c|c|c|c|c|c|c}
			\cline{1-7}
			Type & Graph   & Problem    & \multicolumn{2}{c|}{$K/b/c$}  & Complexity & Ref.  &   \\
			\cline{1-7}
			Del. nodes &  \multirow{3}{*}{  \makecell[c]{General \\ graph} }   &  \multirow{5}{*}{ \makecell[c]{NIP on \\ shortest path } }  &    \multicolumn{2}{|c|}{\multirow{3}*{any $K$}}   &  $\mathcal{NP}$-hard              & \cite{ref4}&   \\
			\cline{1-1}\cline{6-7}
			\multirow{4}{*}{Del. edges}  &  &           &  \multicolumn{2}{|c|}{}        &  $\mathcal{NP}$-hard             &  \cite{ref6}     &   \\
			\cline{6-7}
			&    &       &      \multicolumn{2}{|c|}{}     & Not approx. within  2       & \cite{ref7}      &   \\
			\cline{2-2}\cline{4-7}
			&  \multirow{2}{*}{   \makecell[c]{Undirected \\ network} }                      &                                 &                         \multicolumn{2}{|c|}{$b=1$ }                   & $O(mn)$                &   \cite{ref8}    &   \\
			\cline{4-7}
			&                        &                                 &   \multicolumn{2}{|c|}{ $b\geq 2$}                     &  $\mathcal{NP}$-hard              &    \cite{ref9}  &   \\
			\cline{1-7}
			\multirow{4}{*}{\makecell[c]{Upgrad \\ nodes} }  & \multirow{26}{*}{Tree}    & SDIPT-UN                        & \multirow{4}{*}{any $K $} & any $c$                  &  $\mathcal{NP}$-hard               &  \multirow{4}{*}{ \cite{ref12}  }  &   \\
			\cline{3-3}\cline{5-6}
			&                        & MCSDIPT-UN                      &                        & \multirow{3}{*}{$c=1$}   &  $\mathcal{NP}$-hard              &      &   \\
			\cline{3-3}\cline{6-6}
			&                        & SDIPT-UN$_u$                     &                        &                        & $O(n) $                &      &   \\
			\cline{3-3}\cline{6-6}
			&                        & MCSDIPT-UN$_u$                  &                        &                        & $O(n \log n)$             &     &   \\
			\cline{1-1}\cline{3-7}
			\multirow{22}{*}{  \makecell[c]{Upgrad \\ edges} } &                        & \multirow{2}{*}{MCSPIT$_1$}      & \multirow{5}{*}{any $K $} & c=1                    & $O(n) $                &  \multirow{5}{*}{ \cite{ref10}  }    &   \\
			\cline{5-6}
			&                        &                                 &                        & any $c$                  & $O(n^2)$            &       &   \\
			\cline{3-3}\cline{5-6}
			&                        & \multirow{2}{*}{MSPIT$_1$}       &                        & $c=1$                    & $O(n) $                &        &   \\
			\cline{5-6}
			&                        &                                 &                        & any $c$                  & $O(n^2)$            &       &   \\
			\cline{3-3}\cline{5-6}
			&                        & \multirow{3}{*}{MSPIT$_H$}       &                        & any $c$                  &  $\mathcal{NP}$-hard               &        &   \\
			\cline{4-7}
			&                        &                                 & $K=1$                    & $c=1 $                   & $O(n) $        &   \multirow{3}{*}{  \cite{ref14}  \cite{ref11} }     &   \\
			\cline{4-6}
			&                        &                                 & any $K $                 & $c=1  $                  & $O(n K^2)$ &        &   \\
			\cline{3-6}
			&                        & MCSPIT$_{uH}$                      &                        &                        & $O(n^3 \log n) $     &       &   \\
			\cline{3-7}
			&                        & SDIPT-UE$_H$                     & \multirow{2}{*}{any $K$} & \multirow{2}{*}{any $c$} &  $\mathcal{NP}$-hard                  &\multirow{6}{*}{ \cite{ref10}  }   \\
			\cline{3-3}\cline{6-6}
			&                        & MCSDIPT-UE$_H$                 &                        &                        & $\mathcal{NP}$-hard               &        &   \\
			\cline{3-6}
			&                        & SDIPT-UE$_{uH}$                    & \multirow{2}{*}{any $K $} & \multirow{2}{*}{c =1}  &$O(n) $                  &        &   \\
			\cline{3-3}\cline{6-6}
			&                        & MCSDIPT-UE$_{uH}$               &                        &                        &$O(n \log n)$             &       &   \\
			\cline{3-6}
			&                        & SDIPT-UE$_1$                     & \multirow{2}{*}{any $K $} & \multirow{2}{*}{any $c$} & $O(n) $                 &     &   \\
			\cline{3-3}\cline{6-6}
			&                        & MCSDIPT-UE$_1$                &                        &                        & $O(n) $                &       &   \\
			\cline{3-7}
			&                        & SDIPT$_\infty$    & \multirow{8}{*}{any $K $} & \multirow{8}{*}{any $c$} & $O(n) $                 & (\ref{SDIPT_inf_alg})  &   \\
			\cline{3-3}\cline{6-7}
			&                        & SDIPTC$_\infty$   &                        &                        & $O(n)$            & Alg\ref{SDIPTC_inf_alg}  &   \\
			\cline{3-3}\cline{6-7}			
			&                        & MCSDIPT$_\infty$  &                        &                        & $O(n \log n)$             & Alg\ref{MCSDIPT_inf_alg}  &   \\
			\cline{3-3}\cline{6-7}
			&                        & MCSDIPTC$_\infty$ &                        &                        &$O(Nn^2)$             & Alg\ref{MCSDIPTC_inf_alg}  &   \\
			\cline{3-3}\cline{6-7}
			&                        & SDIPT$_{BH}$                       &                        &                        & $O(n) $                 & (\ref{SDIPT_BH_alg})  &   \\
			\cline{3-3}\cline{6-7}
			&                        & SDIPTC$_{BH}$                      &                        &                        &$O(n)$            & Alg\ref{SDIPTC_BH_alg} &   \\
			\cline{3-3}\cline{6-7}
			&                        & MCSDIPT$_{BH}$                     &                        &                        & $O(n \log n)$             & Alg\ref{MCSDIPT_BH_alg}  &   \\
			\cline{3-3}\cline{6-7}
			&                        & MCSDIPTC$_{BH}$                    &                        &                        & $O(n \log n)$            & Alg\ref{MCSDIPTC_BH_alg}  &   \\
			\cline{1-7}
		\end{tabular}
	}
\end{table}

\section{Solve the problems (SDIPTC$_\infty$) and (MCSDIPTC$_\infty$) under weighted $l_\infty$ norm}\label{sec2}

When the weighted $l_\infty$ norm is applied to the upgrade cost, the problems (SDIPTC$_\infty$) and
(MCSDIPTC$_\infty$) under weighted $l_\infty$ norm, 
can be formulated as the following models (\ref{SDIPTC-inf}) and (\ref{MCSDIPTC_inf}), respectively.

\begin{eqnarray}
	&\max \limits_{\tilde{w}} &\tilde{w}(T):= \sum_{t_k \in Y}{\tilde{w}}(P_k)\nonumber\\
	\textbf{(SDIPTC$_\infty$)} & s.t. &   \max_{\ e\in E}c(e)(\tilde{w}(e)-w(e))\le K,\label{SDIPTC-inf} \\ 
	&&\sum_{e\in E}H(\tilde{w}(e),w(e))\le N, \nonumber\\ 
	&& w(e) \leq \tilde{w}(e) \leq u(e), \quad e \in E.\nonumber
\end{eqnarray}

\begin{eqnarray}
	&\min \limits_{\tilde{w}} & C(\tilde{w}):=\max_{ e\in E}c(e) (\tilde{w}(e)-w(e) ) \nonumber \\
	\textbf {(MCSDIPTC$_\infty$) }  & s.t. &  \sum_{t_k \in Y} {\tilde{w}}(P_k) \geq D,\label{MCSDIPTC_inf} \\
	&&\sum_{e\in E}H(\tilde{w}(e),w(e))\le N,\nonumber \\
	&&w(e) \leq \tilde{w}(e) \leq u(e), \quad e \in E\nonumber.
\end{eqnarray}

This section proves the properties of the problem
(\textbf{SDIPTC$_\infty$}) and presents a divide-and-conquer algorithm 
with time complexity of $O(n)$ for it. 
For the relevant minimum cost problem \textbf{(MCSDIPTC$_\infty$)},
this section provides an algorithm with a time complexity of $O(Nn^2)$.

\begin{definition}\cite{ref12} 
	For any $ e \in E$, 
	define $L(e):=\left \{ t_{k} \vert e \in P_{s,t_k},k=1,2, \dots, r \right \}$
	as the set of leaves $t_{k}$ to which $P_{s,t_k}$ passes through $e$.
	If $t_k \in L(e)$, then $t_k$ is controlled by the edge $e$.
\end{definition}

\subsection{An $O(n)$ time algorithm to solve the problem (\textbf{SDIPTC$_\infty$})}

In order to solve problem (\textbf{SDIPTC$_\infty$}), we first consider its sub-problem (SDIPT$_\infty$) 
by deleting the constraint of cardinality. Its mathematical model can be described as follows.

\begin{eqnarray}
	&\max\limits_{\bar{w}}& \bar{w}(T):=\sum_{t_k \in Y}\bar{w}(P_k)\nonumber \\
	\textbf { (SDIPT$_\infty$)} &s.t.&\max_{e\in E}c(e)(\bar{w}(e)-w(e)) \leq K, \label{SDIPT_inf}\\
	&& w(e) \leq \bar{w}(e) \leq u(e), \quad e \in E.\nonumber
\end{eqnarray}

In the problem (\textbf{SDIPT$_\infty$}),
we can maximize the weight of each edge as much as possible under the constraint of
cost $K$ and the upper bound $u(e)$ of the edge weight $u(e)$.
Therefore, we have the following theorem.

\begin{theorem}\label{SDIPT_inf_ans}
	The weight vector 
	\begin{equation}
		\bar w(e)=w(e)+\min \left \{\frac{K}{c(e)},u(e)-w(e) \right \}(e\in E) \label{SDIPT_inf_w}
	\end{equation}
	is an optimal solution to the problem (\textbf{SDIPT$_\infty$}). 
\end{theorem}

\begin{proof}

	\textbf{(1)} Apparently, the solution $\bar w$ defined by (\ref{SDIPT_inf_w})
	satisfies feasibility.
	
	\textbf{(2)} Suppose that $\bar w$ is not the optimal solution,
	and there exists another edge weight vector $w'$ such that $ w'(T)>\bar{w}(T)$.
	Then at least one edge $e_{i} \in E$ satisfies $ w'(e_{i}) >\bar{w}(e_{i}) $.
	We have $w'(e_i)-w(e_i)>\bar{w}(e_i)-w(e_i)= \min \left \{\frac{K}{c(e)},u(e)-w(e) \right \} $,
	which implies $c(e_i)(w'(e_i)-w(e_i))>K$ or $ w'(e_i)>u(e_i)$, 
	either of which contradicts the cost constraint condition or the upper bound constraint condition,
	respectively.
	Therefore, $\bar w$ is an optimal solution to problem (\textbf{SDIPT$_\infty$}).
 \hb
\end{proof}

Based on Theorem \ref{SDIPT_inf_ans}, for any given $T,n,w,u,c,K$, we can obtain an optimal solution $\bar{w}$ 
and its corresponding optimal value $\bar{w}(T)$ for problem (\textbf{SDIPT$_\infty$}).
For convenience, we denote this subroutine as:
\begin{eqnarray} [ \bar{w},\bar{w}(T) ]=SDIPT_{\infty}(T,n,w,u,c,K) \label{SDIPT_inf_alg}\end{eqnarray}
We just update every edges in the subroutine,
therefore, we have the following conclusion.

\begin{theorem}
	Problem (\textbf{SDIPT$_\infty$}) can be solved by subroutine (\ref{SDIPT_inf_alg}) in $O(n)$ time.
\end{theorem}

Now we can come to the problem (\textbf{SDIPT$_\infty$}) with cardinality constraint.
First, call $[ \bar{w},\bar{w}(T) ]:=SDIPT_{\infty}(T,n,w,u,c,K)$.
Then, we define $\bar{S}(e)=\vert L(e) \vert (\bar w(e)-w(e))(e\in E)$,
which means the increment  of SRD for $e$ under $\bar w$.
Based on that, we can select the $N$ edges with the largest $\bar{S}(e)$, 
and update them to obtain an optimal solution. 

Therefore, in the first step, we use the algorithm proposed by Thoms (2009, pp.220-222) \cite{ref13} 
called $\bar{S}^N:=Selection(\bar{S},N)$ to find the $N$-th largest element
$\bar{S}^N$ in the set $\bar{S}:=\left \{ \bar{S}(e) \vert e \in E \right \} $,
and thus divide the edge set into two parts, $\bar E_N:=\{ e\in E  \vert  \bar{S}(e) \ge \bar{S}^N \} $ and $E\backslash \bar E_N$.

\begin{theorem} \label{SDIPTC_inf_thm}
	The weight vector	\begin{equation}\label{SDIPTC_inf_w}
		\tilde{w}(e)= \begin{cases}
			\bar w(e), & e\in \bar E_N \\
			w(e),& \text{otherwise }\\
		\end{cases}    
	\end{equation}is an optimal solution of the problem (\textbf{SDIPTC$_\infty$}).
\end{theorem}

\begin{proof}
	
	Feasibility is obviously satisfied. 
	Suppose $\tilde{w}$ is not optimal, but $\hat{w}$ is. Then 
	${\hat{w}}(T)>{\tilde{w}}(T)$, and there exists at least one edge $e_k$, such that $\hat{w}(e_k)> \tilde{w} (e_k)$.
	
	\textbf{(1)} If $e_k \in \bar E_N$, then
	$c(e_k)(\hat{w}(e_k)-w(e_k))>c(e_k)(\tilde{w}(e_k)-w(e_k))=\min \left \{ u(e_k),K \right \}$,
	which contradicts the upper bound or update cost constraints.
	
	\textbf{(2)} If $e_k \in E\backslash \bar E_N$, 
	then $\bar w(e_k)\geq \hat w(e_k)>\tilde w(e_k)=w(e_k)$.
	Moreover, there must exist an edge $\hat e$ in $\bar E_N$ 
	such that $\hat w(\hat e)=w(\hat e)<\tilde{w}(\hat e)= \bar w(e_k)$.
	Since $\bar{S}(\hat e)>\bar{S}(e_k)$, there exists a feasible solution $\breve w(e)$ given by
	\begin{eqnarray}
		\breve w(e)= \begin{cases}                \nonumber
			w(e),      &e=e_k\\ 
			\bar w(e), &e= \hat e\\
			\hat w(e),& \text{otherwise}
		\end{cases}
	\end{eqnarray}
	thus we have
	\begin{eqnarray} \nonumber
		&&\breve w(T) -w(T)=\sum_{e \in E\backslash \{e_k,\hat e\}} \vert L(e) \vert (\hat w(e)-w(e))+\bar{S}(\hat e)\\ \nonumber
		&>&\sum_{e \in E\backslash \{e_k\}} \vert L(e) \vert (\hat w(e)-w(e))+\bar{S}(e_k)\\ \nonumber
		&\geq &\sum_{e \in E\backslash \{e_k\}} \vert L(e) \vert (\hat w(e)-w(e))+ \vert L(e_k) \vert (\hat w(e_k)-w(e_k))\\ \nonumber
		&=&\hat w(T) -w(T), \nonumber
	\end{eqnarray}
	this indicates that $\breve w(T) >\bar w(T) $, which 
	contradicts the optimality of $\hat w$. \hb
\end{proof}

Based on theorem \ref{SDIPTC_inf_thm}, we present Algorithm \ref{SDIPTC_inf_alg} below to solve the problem (\textbf{SDIPTC$_\infty$}).

\begin{algorithm}
	\caption{$\left [\tilde{w},\tilde{w}(T) \right ]:=$SDIPTC$_{\infty}(T,n,w,u,c,K,N)$}
	\label{SDIPTC_inf_alg}
	\small
	\begin{algorithmic}[1]
		\REQUIRE A rooted tree $T(V,E)$, the number $n$ of edges, the weight vector $w$, an upper bound vector $u$, 
		a cost vector $c$, and given $K$ and $N$.
		\ENSURE The optimal vector $\tilde{w}$ and the corresponding SRD $ \tilde{w}(T)$.
		\STATE  Call $ [\bar w,\bar w(T)] := SDIPT_{\infty}(T,n,w,u,c,K)$.
		\STATE  For each $e$, determine the set $L(e)$, and calculate $\bar S(e):= \vert L(e) \vert (\bar w(e)-w(e))$.
		\STATE  $\bar{S}^N:=Selection(\bar{S}, N), \bar E_N:=\left \{ e\in E  \vert  \bar{S}(e) \ge \bar{S}^N \right \}$.
		\STATE  Compute $\tilde{w}$ by   (\ref{SDIPTC_inf_w}) and $\tilde{w}(T) := \sum_{t_k \in Y}\tilde{w}(P_k)$ .
		\STATE  \textbf{return} $[\tilde{w},\tilde{w}(T)]$.
	\end{algorithmic}
\end{algorithm}

\begin{theorem}\label{thm-complexity-SDIPTC-infty}
	Problem (\textbf{SDIPTC$_\infty$}) can be solved by Algorithm \ref{SDIPTC_inf_alg} in $O(n)$ time.
\end{theorem}

\begin{proof}
	In Algorithm \ref{SDIPTC_inf_alg}, line 1 solves problem (\textbf{SDIPT$_{\infty}$}) in linear time,
	the Selection algorithm in Line 3 also takes linear time\cite{ref13}, 
	and Line 4 can determine the optimal solution $\tilde w$ in $O(n)$ time.
	Therefore, problem (\textbf{SDIPTC$_\infty$}) can be solved by Algorithm 
	\ref{SDIPTC_inf_alg} in $O(n)$ time. \hb
\end{proof}

\subsection{An $O(Nn^2)$ time algorithm to solve the problem \textbf{(MCSDIPTC$_\infty$)}}

Solving problem \textbf{(MCSDIPTC$_\infty$)} directly is challenging, and therefore, we first focus on solving its sub-problem (\textbf{MCSDIPT$_\infty$}) without the cardinality constraint. This sub-problem can be described as follows:

\begin{eqnarray}
	&\min\limits_{\bar{w}}& C(\bar w):=\max _ { e \in E} c(e)(\bar{w}(e)-w(e))\nonumber \\
	\textbf {(MCSDIPT$_\infty$) } &s.t.&\sum_{t_k \in Y} \bar{w}(P_k) \geq D,\label{MCSDIPT_inf} \\
	&&w(e) \leq \bar{w}(e) \leq u(e), \quad e \in E\nonumber.
\end{eqnarray}

For convenience, let $\Delta{D}:=D-w(T)$.

\subsubsection{An $O(n \log n )$ algorithm to solve (\textbf{MCSDIPT$_\infty$})}

For the  problem (\textbf{MCSDIPT$_\infty$}),
we can solve two special cases based on the relationship between
$u(T)$ and $D$, as well as $w(T)$ and $D$.

\begin{lemma}
	If $D \le w(T)$, then $w$ is an optimal solution to  (\textbf{MCSDIPT$_\infty$}).
	
\end{lemma}

\begin{lemma}
	If $u(T) < D$, then the problem (\textbf{MCSDIPT$_\infty$}) is infeasible.
\end{lemma}
The proofs of the above two lemmas are obvious.
Then we consider the situation that  $w(T) < D \le u(T)$.
First, we will prove an interesting property of the optimal solution. 
\begin{lemma} \label{lemma-sum=D}
	Let $\bar{w}$ be an optimal solution of problem (\textbf{MCSDIPT$_\infty$}),
	then $\bar w(T)=\sum_{t_k\in Y} \bar w(P_k)=D$.
\end{lemma}
\begin{proof}
	Suppose  SRD under optimal solution 
	$\bar{w}$ is $\bar w(T)=D+d(d>0)$.
	Let $\bar w'(e)=\bar w(e) -\frac{d}{c(e)\sum_{e_i \in E}\frac{ \vert L(e_i) \vert }{c(e_i)} }(e \in E$),
	then  
	\begin{eqnarray} \nonumber 
		\bar w'(T) & = & \sum_{t_k\in Y} \bar w'(P_k)  \\ \nonumber 
		&=& \sum_{e \in E} \vert L(e) \vert (\bar w'(e)-\bar w(e))+\sum_{e \in E} \vert L(e) \vert \bar w(e)\\\nonumber 
		&=& -\sum_{e \in E}\frac{d}{c(e)\sum_{e_i \in E}\frac{ \vert L(e_i) \vert }{c(e_i)} } \vert L(e) \vert  + \bar w(T)\\\nonumber 
		&=& -\frac{d}{\sum_{e_i \in E}\frac{ \vert L(e_i) \vert }{c(e_i)}}\sum_{e \in E}\frac{ \vert L(e) \vert }{c(e)}+D+d\\ \nonumber 
		&=& D \\ \nonumber 
	\end{eqnarray}
	while the maximum cost 
	$$C(\bar{w}')=\max_{e \in E}c(e)(\bar w'(e)-w(e))<\max_{e \in E}c(e)(\bar w(e)-w(e))=C(\bar{w})$$
	which contradicts the optimality of $\bar w$. \hb
\end{proof}

Here we construct an optimal solution to the problem (\textbf{MCSDIPT$_\infty$}) 
where $w(T) < D \le u(T)$.
Using the property of the $l_\infty$ norm, 
we know that when the update costs are fixed,
updating every edge can increase  SRD
by the maximum amount. For each edge $e \in E$, 
let $F(e) = c(e)(u(e)-w(e))$
represent the total cost required
to achieve the maximum possible update.
Then we sort the $F(e)$ values of all edges in a non-decreasing order
and renumber them to obtain the sequence $\left \{ F(e_{m_k}) \right \} $ as follows:

\begin{eqnarray}\label{eq-sort-Fe}
	F(e_{m_{1}}) \leq F(e_{m_{2}}) \leq \dots \leq F(e_{m_{n}}).
\end{eqnarray}

For an index $k$, define a function 
\begin{eqnarray}
	f(k)=\sum_{e\in E} \vert L(e) \vert \cdot\min\left\{u(e)-w(e),\frac{F(e_{m_k})}{c(e)}\right\}\label{eq-fx}
\end{eqnarray}
to represent the maximum increment of SRD 
when the update cost is $F(e_{m_k})$

Next, using the binary search method on the sequence $\left \{F(e_{m_k}) \right \}$,
we can determine a critical index $k^*$ such that $f(k^*)\leq D-w(T)\leq f(k^*+1)$.
This allows us to obtain a unique optimal solution to the problem 
(\textbf{MCSDIPT$_\infty$}).
\begin{theorem} \label{MCSDIPT_inf_theorem}
	When $w(T) < D \le u(T)$, let $E_\leq:=\{e\in E \vert F(e)\leq F(e_{m_{k^*}})\}$, the weight vector
	\begin{equation} 
		\bar{w}(e)=\begin{cases} \label{MCSDIPT_inf_w}
			u(e),& e\in E_\leq\\
			w(e)+\frac{F(e_{m_{k*}})}{c(e)} 
			+\frac{\Delta{D} - f(k^*) } {c(e) \sum_{j=k^*+1}^{n} \frac{ \vert L(e_{m_j}) \vert }{c(e_{m_j})}  },& 
			{ e\in E\backslash E_\leq}\\
		\end{cases} 
	\end{equation}
	is the unique optimal solution to the problem (\textbf{MCSDIPT$_\infty$}),  
\end{theorem}

\begin{proof}
	\textbf{(1)} \textbf{First, we prove that $\bar{w}$ is a feasible solution to problem \textbf{MCSDIPT$_\infty$}. }
	
	\textbf{(1.1)} We prove that $\bar w(e)$ for $e \in E$ 
	satisfies the bound constraint. When $e\in E_\leq$,
	we have $\bar w(e)=u(e)$; and when $e\in E\backslash E_\leq$, we have $\frac{F(e_{m_{k^*}})}{c(e)}<u(e)-w(e)$.Then
	\begin{eqnarray*}
		&& f(k^*+1)- f(k^*)\\
		&=&\sum_{i=1}^n\min\{F(e_{m_i}),F(e_{m_{k^*+1}})\}\frac{ \vert L(e_{m_i}) \vert }{c(e_{m_i})}\\&&-\sum_{i=1}^n\min\{F(e_{m_i}),F(e_{m_{k^*}})\}\frac{ \vert L(e_{m_i}) \vert }{c(e_{m_i})}\\
		&=&\left\{\sum_{i=1}^{k^*}F(e_{m_i})\frac{ \vert L(e_{m_i}) \vert }{c(e_{m_i})}+\sum_{i=k^*+1}^{n}F(e_{m_{k^*+1}})\frac{ \vert L(e_{m_i}) \vert }{c(e_{m_i})}\right\}\\
		&&-\left\{\sum_{i=1}^{k^*}F(e_{m_i})\frac{ \vert L(e_{m_i}) \vert }{c(e_{m_i})}+\sum_{i=k^*+1}^{n}F(e_{m_{k^*}})\frac{ \vert L(e_{m_i}) \vert }{c(e_{m_i})}\right\}\\
		&=&\bigg(F(e_{m_{k^*+1}})-F(e_{m_{k^*}})\bigg)\sum_{i=k^*+1}^n\frac{ \vert L(e_{m_i}) \vert }{c(e_{m_i})}.
	\end{eqnarray*}
	Therefore 
	\begin{eqnarray*}
		&&c(e)(\bar{w}(e)-w(e))=\frac{\Delta{D} - f(k^*) } { \sum_{i=k^*+1}^{n} \frac{ \vert L(e_{m_i}) \vert }{c(e_{m_i})}}+F(e_{m_k^*})\\
		&<&\frac{f(k^*+1)-f(k^*)}{\sum_{i=k^*+1}^{n} \frac{ \vert L(e_{m_i}) \vert }{c(e_{m_i})}} +F(e_{m_{k^*}})\\
		& =&\frac{\bigg(F(e_{m_{k^*+1}})-F(e_{m_{k^*}})\bigg)\sum_{i=k^*+1}^n\frac{ \vert L(e_{m_i}) \vert }{c(e_{m_i})}}{\sum_{i=k^*+1}^{n} \frac{ \vert L(e_{m_i}) \vert }{c(e_{m_i})}}+F(e_{m_{k^*}})\\
		& =&F(e_{m_{k^*+1}})\leq F(e).
	\end{eqnarray*}
	Hence, $w(e)\leq \bar w(e)<u(e)$, for any $e\in E\backslash E_\leq$.

	\textbf{(1.2)} Prove that under $\bar w$, $\bar w(T)  \ge D$.  
	\begin{eqnarray}
		&&\bar{w}(T)-w(T)=\sum_{t_k\in Y} \bar w(P_k)-\sum_{t_k\in Y} w(P_k)=\sum_{e\in E} (\bar w(e)-w(e)) \vert L(e) \vert \nonumber\\
		&=&\sum_{e\in E_\leq } (u(e)-w(e)) \vert L(e) \vert \nonumber+ \sum_{e\in E\backslash E_\leq } \bigg(\frac{F(e_{m_k*})}{c(e)}	+\frac{\Delta{D} - f(k^*) } { \sum_{j=k^*+1}^{n} \frac{ \vert L(e_{m_j}) \vert }{c(e_{m_j})} c(e) }\bigg) \vert L(e) \vert \nonumber\\
		&=&f(k^*)+\frac{\big(\Delta{D} - f(k^*)\big)\cdot \sum_{j=k^*+1}^{n} \frac{ \vert L(e_{m_j}) \vert }{c(e_{m_j})}  } { \sum_{j=k^*+1}^{n} \frac{ \vert L(e_{m_j}) \vert }{c(e_{m_j})} }\nonumber\\
		&=&\Delta{D} \label{os-alg2-eq} .\nonumber
	\end{eqnarray}
	
	Therefore,  $\bar w(T)= D$, and $\bar{w}$ is a feasible solution to problem \textbf{MCSDIPT$_\infty$}.
	
	\textbf{(2)} \textbf{Next, we prove the optimality of $\bar{w}$.}
	
	Assume that $\bar{w}$ is not the optimal solution of problem (\textbf{MCSDIPT$_\infty$}), 
	but $w'$ is.
	Then there exists $\bar e \in E$ such that
	$C(\bar{w})=c(\bar{e})(\bar{w}(\bar e)-w(\bar e))>C(w')\ge c(\bar e)(w'(\bar e)-w(\bar e))$.
	Therefore, the total increment in SRD
	due to edge $\bar e$ is $ \vert L(\bar e) \vert (w'(\bar e)-w(\bar e))< \vert L(\bar e) \vert (\bar{w}(\bar e)-w(\bar e))$.
	According to Lemma \ref{lemma-sum=D}, and $\bar w(T)=w'(T)= D$,
	there exists an edge $\hat e\in E$ such that $ \vert L(\hat e) \vert (w'(\hat e)-w(\hat e))> \vert L(\hat e) \vert (\bar{w}(\hat e)-w(\hat e))$,
	which implies that $w'(\hat e)>\bar{w}(\hat e)$.
	
	\textbf{(2.1)} If $\hat e\in E_\leq$, then $w'(\hat e) > \bar w(\hat e)=u(\hat e)$, which contradicts the feasibility of $w'$.
	
	\textbf{(2.2)} If $\hat e\in E\backslash E_\leq$, then $C(w')\geq c(\hat e)(w'(\hat e)-w(\hat e)) > c(\hat e)(\bar{w}(\hat e)-w(\hat e))=C(\bar w)$, which contradicts $C(w')<C(\bar w)$.
	
	Therefore, $\bar w$ is the optimal solution to the problem (\textbf{MCSDIPT$_\infty$}).
	
	\textbf{(3)} \textbf{We prove the uniqueness of $\bar w$ by contradiction}
	
	Suppose $\hat{w}$ is another optimal solution of the problem (\textbf{MCSDIPT$_\infty$}), 
	such that $C(\hat w)=C(\bar w)$ and there exists $e_a \in E$ such that 
	$\hat{w}(e_a) \ne \bar{w}(e_a)$.
	
	\textbf{(3.1)} If $e_a \in E_\leq$, then we have $\hat w(e_a)<\bar w(e_a) =u(e_a)$,
	and we can construct a feasible solution 
	$\hat{w}'(e)=
	\begin{cases}
		u(e),& e =e_a \\
		\hat w (e),& \text{otherwise} \\
	\end{cases} $, 
	where $\sum_{t_k\in Y} \hat w'(P_k)=D+ \vert L(e_a) \vert (u(e_a)-\hat w(e_a))>D$.
	Thus we construct another feasible solution 
	$\hat w''(e)=	\hat w'(e) -\frac{(u(e_a)-\hat w(e_a)) \vert L(e_a) \vert }{c(e)\sum_{e_i \in E}\frac{ \vert L(e_i) \vert }{c(e_i)} },e \in E $,
	we get $ C(	\hat w'') <C(	\hat w')=\max\{C(\hat w),F(e_a)\}=\max\{C(\bar w),F(e_a)\}=C(\bar w)$,
	which contradicts the optimality of $\bar w$.

	\textbf{(3.2)} If $ e_a \in E\backslash E_\leq$ and $ \hat w (e_a)>\bar w(e_a)$,
	then $C(\hat w)\ge c(e_a)(\hat w(e_a)-w(e_a))>c(e_a)(\bar w(e_a)-w(e_a))=C(\bar w)$, which contradicts to the optimality of $\hat w$.
	Suppose $ \hat w (e_a) < \bar w(e_a)$, similarly we can construct a feasible solution
	$  
	\hat{w}'(e)=\begin{cases}
		\bar w(e_a),& e =e_a \\
		\hat w (e),& \text{otherwise}  \\
	\end{cases} 
	$,  where $\sum_{t_k\in Y} \hat w'(P_k)=D+ \vert L(e_a) \vert (\bar w (e_a)-\hat w(e_a))>\sum_{t_k\in Y} \bar w(P_k)=D$.
	
 Then we construct $\hat w''(e)=\hat w'(e) -\frac{ \vert L(e_a) \vert (\bar w (e_a)-\hat w(e_a))}{c(e) \sum_{e_i \in E}\frac{ \vert L(e_i) \vert }{c(e_i)} }(e \in E )$,
	now we have $ \hat w''(T)=D$ and $ C(\hat w'')<C(\hat w') = C(\bar w)$, which contradicts to the optimality of $\bar w$.
	
	In conclusion, the weight vector defined by   (\ref{MCSDIPT_inf_w}) is the unique optimal solution to the problem (\textbf{MCSDIPT$_\infty$}).  \hb
\end{proof}

Based on Theorem \ref{MCSDIPT_inf_theorem}, we present Algorithm \ref{MCSDIPT_inf_alg} 
for problem (\textbf{MCSDIPT$_\infty$}).
\begin{algorithm}
	\caption{$\left [\bar{w},C(\bar w) \right ]:=$MCSDIPT$_{\infty}(T,n,w,u,c,D)$}
	\label{MCSDIPT_inf_alg}
	\small
	\begin{algorithmic}[1]
		\REQUIRE A rooted tree $T(V,E)$, the number $n$ of edges, the weight vector $w$, an upper bound vector $u$, 
		a cost vector $c$, and given value $D$.
		\ENSURE The optimal vector $\bar{w}$ and the cost  $C(\bar w)$.
		\STATE  For each edge $e \in E $, determine the set $L(e)$ and calculate $F(e):=c(e)(u(e)-w(e))$, 
		$u(T):=\sum_{t_k\in Y} u(P_k)$, and let $\Delta{D}:=D-w(T)$
		\IF{${u}(T)<D$} 
		\STATE  \textbf{return} ``The problem is infeasible" and $\bar{w}=[\;], C(\bar{w})= + \infty$ 
		\ELSE
		\IF{$w(T)>D$}
		\STATE  \textbf{return} $[w,0]$.
		\ELSE
		\STATE  Sort $F(e)$ in a non-decreasing order by (\ref{eq-sort-Fe}).
				
		\STATE  Initialize $a:=1$, $b:=n$, $k^*:=0$.
		\WHILE{$b-a >1$ and $k^*=0$}
		\STATE  $k:=\left \lfloor \frac{a+b}{2} \right \rfloor$.
		\STATE  Calculate the values of $f(k)$ and $f(k+1)$ using   (\ref{eq-fx}).
		\IF{$f(k)\leq \Delta D  \leq f(k+1)$}
		\STATE  $k^*:= k$.
		\ELSIF{$f(k+1)< \Delta D$}
		\STATE  $a := k$.
		\ELSE
		\STATE  $b := k$.
		\ENDIF
		\ENDWHILE
		\STATE  Calculate $\bar w$ using   (\ref{MCSDIPT_inf_w}), $ C(\bar w):=F(e_{m_{k^*}})
		+\frac{\Delta{D} - f(k^*) } { \sum_{j=k^*+1}^{n} \frac{ \vert L(e_{m_j}) \vert }{c(e_{m_j})}  }$.
		\STATE  \textbf{return} $[\bar{ w}, C({\bar w})]$.
		\ENDIF
		\ENDIF
	\end{algorithmic}
\end{algorithm}

\begin{theorem}\label{MCSDIPT_inf_time}\label{thm-complexity-MCSDIPT-infty}
	The problem (\textbf{MCSDIPT$_\infty$}) can be solved by Algorithm \ref{MCSDIPT_inf_alg}  in $O(n\log n)$ time.
\end{theorem}
\begin{proof}
	The sorting in Line 8 can be completed in $O(n\log n)$ time.
	The binary search process in Lines 10-21 takes $O(\log n)$ iterations.
	So it takes a total of $O(n\log n)$ time. Therefore, 
	the problem (\textbf{MCSDIPT$_\infty$}) can be solved
	 in $O(n\log n)$ time. \hb
\end{proof}

\subsubsection{An $O(N n^2 )$ algorithm to solve \textbf{(MCSDIPTC$_\infty$)}}

For the original problem \textbf{(MCSDIPTC$_\infty$)},
we first consider its feasibility. 
We define 
$S(e):= \vert L(e) \vert ( u(e)-w(e))(e\in E)$, 
which represents the maximum increment
in SRD
that each edge can achieve.
We sort the edges in a non-increasing order of $S(e)$ and renumber them to obtain the following sequence:
\begin{equation}\label{eq-sort-Se}
	S({e_{\beta_{1}}}) \geq S({e_{\beta_{2}}}) \geq \dots \geq S({e_{\beta_{n-1}}}) \geq S({e_{\beta_{n}}})
\end{equation}
Under the cardinality constraint of $N$ edges, the maximum
total increment in
SRD is $S_N=\sum_{i=1}^{N}S({e_{\beta_{i}}})$.

If $S_N < \Delta D$, then it is clear that
the problem has no solution, as the increment $S_N$
is not sufficient to meet the problem requirements.
\begin{lemma}\label{lem-infea-MCSDIPTC}
	When $S_N < \Delta D$, problem \textbf{(MCSDIPTC$_\infty$)} is infeasible.
\end{lemma}

When $S_N \geq \Delta D$, we give the following lemma 
to determine the number of edges that are updated in
any optimal solution.

\begin{lemma}
Suppose $S_N \geq \Delta D$, and let $\tilde{w}$ be an optimal solution of problem \textbf{(MCSDIPTC$_\infty$)}. Let $R = \sum_{e\in E}H(w(e),u(e))$. Then, we have
$$
\sum_{e\in E}H(\tilde{w}(e),w(e)) = \min \{R,N\}.
$$
\end{lemma}

\begin{proof}
	(1) When $R \le N$, it means that the cardinality constraint is always satisfied. So the problem 
	\textbf{(MCSDIPTC$_\infty$)} can be directly transformed to its sub-problem (\textbf{MCSDIPT$_\infty$}).
	According to Theorem \ref{MCSDIPT_inf_theorem}, all $R$ edges will be updated, so $\sum_{e\in E}H(\tilde{w}(e),w(e))=R$ in this case.
	
	(2) When $R>N$, we prove 
	$\sum_{e\in E}H(\tilde{w}(e),w(e))=N$ by contradiction.
	Suppose that
	$\sum_{e\in E}H(\tilde{w}(e),w(e))<N$. Then, there is an edge $e_i$ such that $\tilde w(e_i)=w(e_i)<u(e_i)$. We can construct a solution $
	\hat{w}(e)=    
	\begin{cases}
		\tilde{w} (e),  & e \ne e_i\\
		w(e)+ \varepsilon,  & e=e_i
	\end{cases}   
	$,
	where $0<  \varepsilon  \le  \min  \left \{   \frac{ C(\tilde w)}{c(e_i)},u(e_i)-w(e_i)) \right \}  $. Then $C(\hat w) \leq C(\tilde w)$ and thus
	\begin{eqnarray}
		&& \hat w(T)= \sum_{t_k \in Y}{\hat {w}}(P_k)= \sum_{e \in E} \vert L(e) \vert \hat {w}(e)\nonumber\\
		&=&\sum_{e \in E\backslash\{e_i\}} \vert L(e) \vert \tilde {w}(e) + \vert L(e_i) \vert  (\tilde {w}(e_i)+ \varepsilon) \nonumber\\
		&=&\sum_{e \in E} \vert L(e) \vert \tilde {w}(e) + \vert L(e_i) \vert \cdot  \varepsilon \nonumber\\
		&=&\tilde{w}(T)+ \vert L(e_i) \vert \cdot  \varepsilon .\label{eq-temp0}
	\end{eqnarray}
	
	Let $\tilde{E} =\left \{ e \vert c(e) (\tilde{w}(e)-w(e) )=C(\tilde w)\right \}$, we can construct a solution
	$$  \bar {w}(e)=    
	\begin{cases}
		\hat{w} (e),  & e \notin  \tilde{E}\\
		\hat{w}(e)-\frac{ \varepsilon   \vert L(e_i) \vert  }{\sum_{e_j \in \tilde{E} } \vert L(e_j) \vert }, &e\in \tilde{E}
	\end{cases}$$
	then we have
	\begin{eqnarray*}
		&&\bar w(T) =\sum_{t_k \in Y}{\bar {w}}(P_k)=\sum_{t_k \in Y}{\hat {w}}(P_k)-\frac{ \varepsilon   \vert L(e_i) \vert  }{\sum_{e_j \in \tilde{E} } \vert L(e_j) \vert }\sum_{e \in \tilde{E} } \vert L(e) \vert \\
		&=&\hat {w}(T) - \varepsilon   \vert L(e_i) \vert  =\tilde{w}(T) \ge D,
	\end{eqnarray*}
	where the last equality can be obtained from (\ref{eq-temp0}).
	Furthermore,  note that  
	$$\sum_{e \in E}H(\bar w(e),w(e))=\sum_{e \in E}H(\tilde w(e),w(e)) +1 \le N,$$
	then $\bar{w}$ is a feasible solution.
	However,  it follows from the definitions of $ \hat w$ and $ \bar w$ that  $C(\tilde{w} )>C(\bar w)$, which contradicts the optimality of
	$\tilde{w}$. \hb
\end{proof}

Based on the analysis above, if we can determine the set $\bar E^*$ of $N$
edges that are updated in the optimal solution, then the problem \textbf{(MCSDIPTC$_\infty$)} can 
be transformed into its sub-problem \textbf{(MCSDIPT$_\infty$)}
which can be solved by Algorithm \ref{MCSDIPT_inf_alg} in $O(n \log n)$ time.
Next, we use an iterative method to determine the set $\bar E^*$.

Initialization: $\tau:=0$, $\bar E^{\tau}:=\{e_{\beta_1}, e_{\beta_2}, \cdots, e_{\beta_N}\}$.

Call Algorithm \ref{MCSDIPT_inf_alg}:
$[\bar{w}^\tau,C(\bar{w}^\tau)]$=MCSDIPT$_{\infty}(T,n,w,u^\tau,c,D)$,
where $u^{\tau}(e)$ $=\left\{\begin{array}{lr}u(e),&e\in \bar E^{\tau}\\ w(e), &\text{otherwise} \end{array}\right.$,
then the obtained $(\bar{w}^\tau, C(\bar{w}^\tau))$ is the initial feasible solution and the corresponding objective value of the problem.

Next, for each edge in the current edge set $\bar{E}^\tau$,
consider whether there exists a better edge in the set $E\backslash \bar E^{\tau}$ that can reduce the upgrade cost by replacing it.

Sort the values $\nu (e):=\frac{ \vert L(e) \vert }{c(e)}$ in a non-increasing order such that
\begin{eqnarray} \label{sequence_gamma}
	\nu (e_{\gamma_{1}}) \ge \nu (e_{\gamma_{2}}) \ge \dots \ge \nu (e_{\gamma_{n}}).
\end{eqnarray}

Let $S^\tau(e):=  \min\{S(e),\nu(e) C(\bar w^{\tau})\}$
denote the maximum SRD increment of 
edge $e$ within the  maximum cost
$C(\bar w^{\tau})$.

Let $\bar{E}^\tau_{=}:=\left \{ e\in E \vert c(e)(\bar{w}^\tau(e)-w(e))=C(\bar{w}^\tau)\right \}$ 
be the set of edges whose costs reach  $C(\bar w^{\tau})$.
For edge $e_i \in \bar{E}^{\tau}$ and edge $e_j \in E \backslash \bar{E}^{\tau}$, 
where $j$ is traversed according to the sequence $\left \{ {\gamma_n} \right \} $, 
there are two replacement scenarios as follows.

(1) If $e_i\in \bar{E}^\tau_{=}$, 
$S^\tau(e_j)=  \min\{S(e_j),\nu (e_j)C(\bar w^{\tau})\}\geq 
\nu (e_i)C(\bar w^{\tau})=S^\tau(e_i) $
and $\nu (e_i)<\nu (e_j)$. 
Then $e_i$ is replaced with 
$e_j$ and update $\bar{E}^{\tau}=\bar{E}^{\tau}\backslash{\left \{ e_i \right \} }\cup{\left \{ e_j \right \} }$.

(2) If $e_i \in \bar{E}^\tau \setminus \bar{E}^\tau_{=}$ and $S^\tau(e_j)>S^\tau(e_i)$,
then $e_i$ is replaced with $e_j$ and 
$\bar{E}^{\tau}=\bar{E}^{\tau}\backslash{\left \{ e_i \right \} }\cup{\left \{ e_j \right \} }$.

After the replacement is done for all $e_i \in \bar E ^{\tau}$and $e_j \in E\setminus \bar E ^{\tau}$, 
set $\bar E^{\tau+1}=\bar E^{\tau}$, the first iteration is 
finished.

In the second iteration, let  $u^{\tau+1}(e)=\left\{\begin{array}{lr}u(e),e\in \bar E^{\tau+1}\\ w(e), \text{otherwise} \end{array}\right.,$ and call
Algorithm \ref{MCSDIPT_inf_alg}: $[\bar{w}^{\tau+1},C(\bar{w}^{\tau+1})]$=MCSDIPT$_{\infty}(T,n,w,u^{\tau+1},c,D)$.
The resulting $\bar{w}^{\tau+1}$ is a better feasible solution to problem \textbf{(MCSDIPTC$_\infty$)} than $\bar{w}^{\tau} $, with a corresponding objective value of $C(\bar{w}^{\tau})$. 

\begin{theorem} \label{MCSDIPTC_theorem_1}
	The solution $\bar{w}^{\tau+1}$ is feasible to problem 
	\textbf{(MCSDIPTC$_\infty$)}, 
	and $C(\bar{w}^\tau)\geq C(\bar{w}^{\tau+1})$.
\end{theorem}

\begin{proof}
	Without loss of generality, let $\bar{E}^{\tau+1}=\bar{E}^{\tau} \setminus \left \{ e_i \right \} \cup  \left \{ e_j \right \},$
	where $e_i\in \bar E^{\tau}$ and $e_j\in E\backslash \bar E^{\tau}$.
	
	\textbf{(1) If $e_i \in \bar{E}^\tau_=$}, then 
	\begin{eqnarray}&& S^\tau(e_j)\geq S^\tau(e_i),\label{th19-pf-case1-1}\\
&&\nu (e_i)<\nu (e_j).
\label{th19-pf-case1-2}\end{eqnarray}
	
	\textbf{(1.1) Use induction to prove the feasibility of $\bar w^{\tau+1}$.} 
	First, since $\sum_{e \in \bar{E}^{0}} S(e) =S_N \geq D$, the vector $\bar{w} ^0$ is a feasible solution to the original problem.
	
	Assume that the feasibility holds when $\tau=k$ then
	$$\sum_{e \in \bar{E}^{\tau+1}} S(e) \ge \sum_{e \in \bar{E}^{\tau+1}} S^\tau (e)\ge \sum_{e\in \bar{E}^\tau }S^\tau(e)=\sum_{e\in \bar{E}^\tau\backslash{\left \{ e_i \right \} }}S^\tau(e)+S^\tau(e_i)\ge \Delta D.$$ 
	which means that $u^\tau (T)-w(T) \ge \Delta D$ and $u^\tau(T ) \ge D$.
	Thus we can get a feasible solution after calling   Algorithm \ref{MCSDIPT_inf_alg}, which satisfies $\bar w ^{\tau +1}(T) \ge D$.
	Furthermore, $ \vert \bar{E}^{\tau+1} \vert = \vert \bar{E}^{\tau} \vert =N$, and for any edge $e\in E$, $w(e)\leq \bar w^{\tau+1}(e)\leq u(e)$,
	so $\bar w^{\tau+1}$ is a feasible solution to problem \textbf{(MCSDIPTC$_\infty$)}.
	
	\textbf{(1.2) We show $C(\bar{w}^\tau) \ge C(\bar{w}^{\tau+1})$ by contradiction.}
	Assume $C(\bar{w}^{\tau}) < C(\bar{w}^{\tau+1})$. Then,
	$ \sum_{e \in \bar{E}^{\tau} \setminus  \left \{ e_i \right \} } S^{\tau}(e) \le \sum_{e \in \bar{E}^{\tau+1} \setminus  \left \{ e_j \right \} } S^{\tau+1}(e)$.
	Also, by Lemma \ref{lemma-sum=D}, we have
	$\sum_{e \in \bar{E}^{\tau}} S^\tau(e)= \Delta D=\sum_{e \in \bar{E}^{\tau+1}} S^{\tau+1}(e)$
	which implies that
	
	$$\sum_{e \in \bar{E}^{\tau} \setminus  \left \{ e_i \right \} } S^\tau(e)+S^\tau(e_i)=\sum_{e \in \bar{E}^{\tau+1} \setminus  \left \{ e_j \right \} } S^{\tau+1}(e)+S^{\tau+1}(e_j)$$
	
	Therefore,
	\begin{eqnarray} 
		S^{\tau}(e_i) \ge S^{\tau+1}(e_j) \ge S^{\tau}(e_j)\label{th19-pf-case1.2-1}
	\end{eqnarray}
	
	Combining equations (\ref{th19-pf-case1-1}) and (\ref{th19-pf-case1.2-1}), we have
	\begin{eqnarray}
		S^{\tau}(e_i)=S^{\tau}(e_j)= S^{\tau+1}(e_j),\label{th19-pf-case1.2-2}
	\end{eqnarray}
	which implies that
	$\sum_{e \in \bar{E}^{\tau} \setminus  \left \{ e_i \right \} } S^{\tau}(e) = \sum_{e \in \bar{E}^{\tau+1} \setminus  \left \{ e_j \right \} } S^{\tau+1}(e)$.
	
	Then, for any edge $e \in \bar{E}^{\tau} \setminus  \left \{ e_i \right \}$,
	$\bar w^{\tau+1}(e)= \bar w^\tau(e)$, combined with the assumption $C(\bar{w}^{\tau}) < C(\bar{w}^{\tau+1})$, we have
	$$C(\bar w^\tau)=c(e_i)(\bar w^\tau (e_i)-w(e_i))<c(e_j)(\bar w^{\tau+1}(e_j)-w(e_j))=C(\bar w^{\tau +1}).$$ 
	It follows from (\ref{th19-pf-case1.2-2}) that
	$$ \nu(e_i)C(\bar w^{\tau})= S^{\tau}(e_i)=S^{\tau+1}(e_j)=  \nu(e_j)C(\bar w^{\tau+1}), $$
	which implies that
	$$\nu(e_i) = \nu(e_j)\frac{C(\bar w^{\tau+1})}{C(\bar w^{\tau})}.$$
	Therefore, we have $\nu (e_i) > \nu (e_j)$, which contradicts equation (\ref{th19-pf-case1-2}).

	\textbf{(2) If $e_i \notin \bar{E}^\tau_=$}, then $S^\tau(e_j)>S^\tau(e_i)$. Using the same argument as \textbf{(1.1)}, we can prove that $\bar w^{\tau+1}$ is a feasible solution for problem \textbf{(MCSDIPTC$_\infty$)}.
	Next, we show $C(\bar{w}^\tau) \ge C(\bar{w}^{\tau+1})$ by contradiction.
	Similar to \textbf{(1.2)}, we have equation (\ref{th19-pf-case1.2-1}).
	However, this contradicts the fact that $S^\tau(e_j)>S^\tau(e_i)$. \hb
\end{proof}

Then, update $\tau:=\tau+1$ and continue iterating until there is no edge 
in $E \setminus \bar E^\tau$ better than any edge in the current set $\bar E^\tau $.
Denote the resulting set of edges by $\bar E^*$.
Let $u^{*}(e)=\left\{\begin{array}{lr}u(e),e\in \bar E^{*}\\
		w(e), \text{otherwise} \end{array} \right..$ Then, call Algorithm \ref{MCSDIPT_inf_alg}:
\begin{eqnarray}\label{os-1}
	[\tilde w,C(\tilde w)]=MCSDIPT_{\infty}(T,n,w,u^*,c,D),
\end{eqnarray}

\begin{theorem}
	The solution $\tilde w$ obtained by (\ref{os-1}) is an optimal solution to problem \textbf{(MCSDIPTC$_\infty$)}.
\end{theorem}

\begin{proof}
	According to Theorem \ref{MCSDIPTC_theorem_1}, the iteratively obtained solution $\tilde w$ is feasible.
	
	Assume that $\tilde w$ is not optimal, but $\hat w$ is, such that 
	$\max_{e \in E}c(e)(\hat{w}(e)-w(e))<\max_{e \in E}c(e)(\tilde{w}(e)-w(e))$. 
	
	Let $\hat{E}_N= { \left \{ e\in E \mid \hat{w}(e) \neq w(e) \right \}  }$,
	$\tilde{E}_N= { \left \{ e\in E \mid \tilde{w}(e) \neq w(e)  \right \} }$,
	$\tilde{E}'= \tilde{E}_N \backslash\hat{E}_N$,
	$\hat{E}'= \hat{E}_N \backslash \tilde{E}_N$,
	and $E_J=\tilde{E}_N \cap \hat{E}_N$.
	\begin{figure}[!htbp]
		\centering
		\includegraphics[width=.4\linewidth]{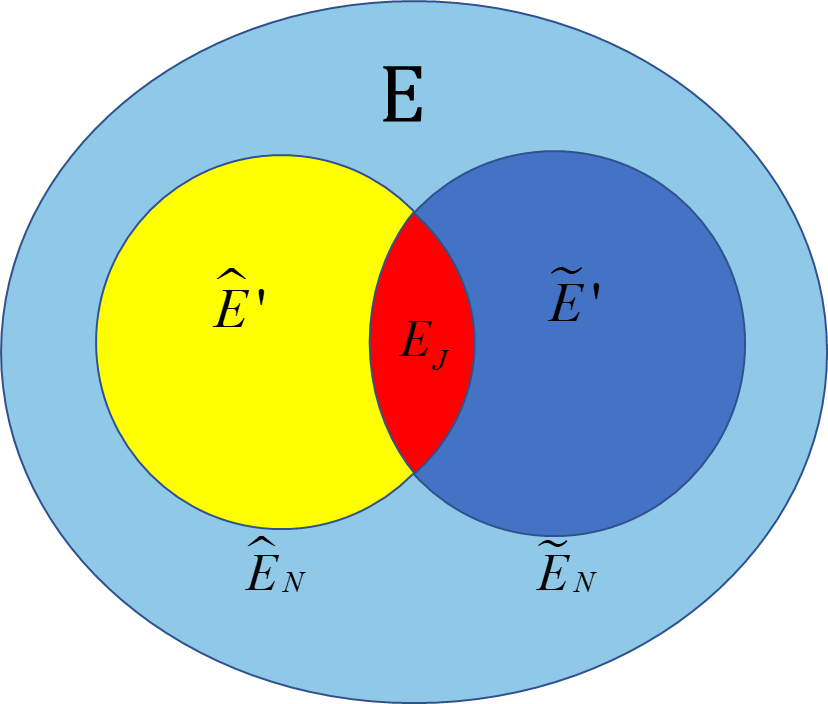}
		\caption{Inter-set relations}\label{graph_3}
	\end{figure}
	
	(1) If $\hat{E}_N =\tilde{E}_N$, then by the uniqueness of the optimal solution of the problem
	(\textbf{MCSDIPT$_\infty$}), we have $\hat{w}=\tilde{w}$, which 
	contradicts the assumption.
	
	(2) If $\hat{E}_N \neq \tilde{E}_N$, since 
	$C(\hat{w})=\max_{e \in E}c(e)(\hat{w}(e)-w(e))<\max_{e \in E}c(e)(\tilde{w}(e)-w(e))=C(\tilde{w})$, 
	we have $\max_{e \in E_J}c(e)(\hat{w}(e)-w(e))\leq \max_{e \in E_J}c(e)(\tilde{w}(e)-w(e))$.
	Therefore, the total increment in SRD on $E_J$ by $\hat{w}$ should be no more than that by $\tilde{w}$ and $\sum_{e \in E_J} \vert L(e) \vert  (\hat{w}(e)-w(e)) \leq \sum_{e \in E_J} \vert L(e) \vert  (\tilde{w}(e)-w(e))$. Moreover, since $\sum_{t_k \in Y}\hat{w}(P_k) \geq D$ and $\sum_{t_k \in Y}\tilde{w}(P_k) = D$, we have
	\begin{eqnarray}\sum_{e \in \hat{E}'} \vert L(e) \vert  (\hat{w}(e)-w(e)) \ge \sum_{e \in \tilde{E}'} \vert L(e) \vert  (\tilde{w}(e)-w(e)).\label{eq-th20-2}\end{eqnarray}
	
	(2.1) If ``$>$'' holds in equation (\ref{eq-th20-2}), then there exist
	$e_a\in \hat{E}'$ and $e_b\in \tilde{E}'$ such that
	$ \vert L(e_a) \vert (\hat{w}(e_a)-w(e_a))> \vert L(e_b) \vert (\tilde{w}(e_b)-w(e_b))$, 
	which implies that $\hat{S}(e_a)= \vert L(e_a) \vert \min\big\{u(e_a)-w(e_a),\frac{C(\hat{w})}{c(e_a)} \big\} 
	> \vert L(e_b) \vert \min\big\{u(e_b)-w(e_b),\frac{C(\tilde{w})}{c(e_b)} \big\} =\tilde{S}(e_b)$.
	
	From the replacement case (2) of the iteration $\bar{E}^\tau$, it can be concluded that at the end of the 
	iteration, for each edge in $\bar{E}$ under $\tilde{w}$, there is no edge in $E\setminus \bar{E}$ 
	that can replace any edge in $\bar E$. However, the contradiction arises from the fact that $\tilde{S}(e_a)\ge \hat S(e_a )>\tilde{S}(e_b)$ and edge $e_a$ satisfies the replacement requirement of edge $e_b$.
	
	(2.2) If the equality holds in   (\ref{eq-th20-2}), it means that the two edge weight 
	vectors achieve the same SRD on the same edge set $E_J$. Due to the uniqueness of the 
	optimal solution of the problem (MCSDIPT$_\infty$), it follows that
	SRD increment by each edge in $E_J$ must be equal.
	Therefore, the maximum cost edges of the two schemes are in $\hat E'$ and $\tilde E'$, respectively. 
	Moreover, since $C(\hat{w})<C(\tilde{w})$, it follows that $\hat E'$ is better than $\tilde E'$, and there must be at least one edge $e_a \in \hat{E}',e_b \in \tilde{E}'$,
	such that $\tilde{S}(e_a)>\tilde{S}(e_b)$ and edge $e_a$ satisfies the replacement requirement of edge $e_b$, which leads to a contradiction.
	
	Therefore, $\tilde{w}$ is the optimal solution to the original problem. \hb
\end{proof}

The following Algorithm \ref{MCSDIPTC_inf_alg} can be obtained from the above analysis to solve the problem
(MCSDIPT$_\infty$).
\begin{algorithm}[!h]
	\caption{$\left [\tilde{w},C(\tilde{w} ) \right ]:=$MCSDIPTC$_{\infty}(T,n,w,u,c,D,N)$}
	\label{MCSDIPTC_inf_alg}
	\small
	\begin{algorithmic}[1]
		\REQUIRE A rooted tree $T(V,E)$, the number $n$ of edges, the weight vector $w$, an upper bound vector $u$, 
		a cost vector $c$, and given value $D$ and $N$.	
		\ENSURE The optimal vector $\tilde{w}$ and the cost $C(\tilde{w})$.
		\STATE  For each $e \in E$, determine the set  $L(e)$, 
		calculate  $S(e):= \vert L(e) \vert (u(e)-w(e))$, and let $\Delta D=D-w(T)$.
		\STATE  Sort the $S(e)$ in a non-increasing order by (\ref{eq-sort-Se}).
		\STATE  $S_N:=\sum_{i=1}^{N}S({e_{\beta_{i}}})$.
		\IF{$S_N < \Delta D$}
		\STATE  \textbf{return} ``The problem is infeasible" and $\tilde{w}=[\;], C(\tilde{w})= + \infty $.
		\ELSE
		\STATE  Sort $\nu (e):=\frac{ \vert L(e) \vert }{c(e)}, e \in E$ in a non-increasing order by (\ref{sequence_gamma}).

		\STATE  $flg:=1, \tau :=0, \bar{E}^\tau:=\left \{ e_{\beta_1}, e_{\beta_2},\dots,e_{\beta_N} \right \}$.
		\STATE  Let $S^\tau(e):=  \min\{S(e),\nu(e) C(\bar w^{\tau}) \}, e \in E .$
		\WHILE{$flg=1$}
		\STATE  $flg:=0$.
		\STATE  Let $u^{\tau}(e):=\left\{\begin{array}{lr}u(e),e\in \bar E^{\tau}\\ w(e), \text{otherwise} \end{array}\right.$. 
        \STATE  $[\bar{w}^\tau,C(\bar{w}^\tau)]$=MCSDIPT$_{\infty}(T,n,w,u^{\tau},c,D)$.
		 
		\STATE  $\bar{E}^\tau_{=}:=\left \{ e\in E \vert c(e)(\bar{w}^\tau(e)-w(e))=C(\bar{w}^\tau)\right \}$.
		\FOR{$e_i \; \in \;\bar{E}^\tau$  }
		\FOR{$e_j \; \in \; E\setminus\bar{E}^\tau$ in the order by (\ref{sequence_gamma})}
		\IF{$e_i\in \bar{E}_=^\tau$}
		\IF{$S^\tau(e_j)\ge S^\tau(e_i) \;  and \;\nu(e_i) < \nu (e_j)$}
		\STATE  $\bar{E}^{\tau}:=\bar{E}^{\tau}\setminus \left \{ e_i \right \} \cup \left \{ e_j \right \},flg:=1.$  
		\STATE  \textbf{break}
		\ENDIF
		\ELSE
		\IF{$S^\tau(e_j)>S^\tau(e_i)$}
		\STATE  $\bar{E}^{\tau}:=\bar{E}^{\tau}\setminus \left \{ e_i \right \} \cup \left \{ e_j \right \},flg:=1.$        
		\STATE  \textbf{break}
		\ENDIF
		\ENDIF
		\ENDFOR
		\ENDFOR
		\STATE  Update $\bar{E}^{\tau+1}:=\bar{E}^{\tau}$, $\tau:=\tau+1.$
		\ENDWHILE 
		\ENDIF
		\STATE  Update $\tau:=\tau-1.$
	\STATE  \textbf{return [$\tilde{w}:=\bar{w}^\tau,C(\tilde{w}):=C(\bar{w}^\tau)] $.	}
	\end{algorithmic}
\end{algorithm}

\begin{theorem}
	Problem \textbf{(MCSDIPTC$_\infty$)} can be solved by Algorithm \ref{MCSDIPTC_inf_alg} in $O(Nn^2)$ time.
\end{theorem}

\begin{proof}
	In Algorithm \ref{MCSDIPTC_inf_alg}, we sort $S(e)$ and $\nu(e)$ in Lines 2 and 7, which takes a total of $O(n\log n)$ time.
	In the loop from Line 10 to 31, Line 13 calculates the optimal solution $\bar{w}$ and $C(\bar{w})$ for
	the 
	problem \textbf{(MCSDIPT$_\infty$)} obtained by modifying only $N$ edges in $\bar{E}$, which requires $O(n\log n)$ time according to Theorem \ref{MCSDIPT_inf_time}.
	In the loop from Line 15 to 29, we iterate over the sets $\bar E ^\tau$ and $E \setminus \bar E ^\tau$, which takes $Nn$ iterations. In	each iteration we can update the selected edges in $O(n)$ time.
	Therefore, problem \textbf{(MCSDIPT$_\infty$)} can be solved by Algorithm \ref{MCSDIPTC_inf_alg} in $O(Nn^2)$ time.  \hb
\end{proof}

\section{Solve the problems (SDIPTC$_{BH}$) and (MCSDIPTC$_{BH}$) under weighted bottleneck Hamming distance }\label{sec3}

When the weighted bottleneck Hamming distance is applied to the upgrade cost, 
the problems (SDIPTC$_{BH}$) and
(MCSDIPTC$_{BH}$) 
can be formulated as the following models (\ref{SDIPTC_BH_m}) and (\ref{MCSDIPTC_BH}), respectively.

\begin{eqnarray}
	&\max \limits_{\tilde{w}} & \tilde{w}(T):= \sum_{t_k \in Y}{\tilde{w}}(P_k)\nonumber\\
	\textbf{(SDIPTC$_{BH}$)} & s.t. &  \max_{\ e\in E}c(e)H( \tilde{w}(e),w(e))\le K,\label{SDIPTC_BH_m} \\ 
	&&\sum_{e\in E}H(\tilde{w}(e),w(e))\le N, \nonumber\\ 
	&& w(e) \leq \tilde{w}(e) \leq u(e), \quad e \in E.\nonumber
\end{eqnarray}
\begin{eqnarray}
	&\min \limits_{\tilde{w}} & C(\tilde w):= \max_{ e\in E}c(e)H({\tilde{w}(e),w(e))} \nonumber \\
	\textbf {(MCSDIPTC$_{BH}$) } & s.t. & \sum_{t_k \in Y} {\tilde{w}}(P_k) \geq D,\label{MCSDIPTC_BH} \\
	&&\sum_{e\in E}H(\tilde{w}(e),w(e))\le N,\nonumber \\
	&&w(e) \leq \tilde{w}(e) \leq u(e), \quad e \in E\nonumber.
\end{eqnarray}

In this section, we first provide an algorithm with time complexity $O(n)$ for 
problem (\textbf{SDIPTC$_{BH}$}).
For the minimum cost problem  (\textbf{MCSDIPTC$_{BH}$}), 
we transform it into its sub-problem (\textbf{MCSDIPT$_{BH}$}) and provide an 
algorithm with time complexity $O(n\log n)$.

\subsection{An $O(n)$ time algorithm to solve the problem (\textbf{SDIPTC$_{BH}$})} 
Similarly, we first consider solving the sub-problem (\textbf{SDIPT$_{BH}$}) without the cardinality constraint, 
which can be formulated as follows.

\begin{eqnarray}
	&\max\limits_{\bar{w}}&\bar{w}(T):= \sum_{t_k \in Y}{\bar{w}}(P_k)\nonumber \\
	\textbf { (SDIPT$_{BH}$)} & s.t.& \max_{ e\in E}c(e)H({\bar{w}(e),w(e))} \leq K, \label{SDIPT_BH}\\
	&& w(e) \leq \bar{w}(e) \leq u(e), \quad e \in E.\nonumber
\end{eqnarray}

Based on the characteristic of bottleneck Hamming distance, in order to maximize the objective function, we consider updating the edges $e\in E$ that satisfy $c(e)\leq K$. Thus, we give the following optimality condition.

\begin{theorem}\label{SDIPT_BH_ans}
	The weight vector
	\begin{equation}\label{SDIPT_BH_w}
		\bar{w}(e) = 
		\begin{cases}
			u(e), & c(e)\leq K\\
			w(e), & c(e) > K 
		\end{cases}
	\end{equation}
	is an optimal solution to problem (\textbf{SDIPT$_{BH}$}).
\end{theorem}

\begin{proof}
	Clearly, $\bar w$ is a feasible solution to problem (\textbf{SDIPT$_{BH}$}).
	Suppose that $\bar w$ is not the optimal solution, but $w'$ is. 
	Then, there must be at least one edge $ e_{i} \in E $ such that $ w'(e_i)> {\bar{w}(e_{i}) } $.
	If $c(e_i) \leq K$, we have $w'(e_i)\leq u(e)=\bar{w}(e_i)$,
	which contradicts $ w'(e_i)> {\bar{w}(e_{i}) } $. If $c(e_i) > K$, 
	the cost should be at least $c(e_i)$, which is greater than the cost
	limit $K$, contradicting the assumption. Therefore, the $\bar w$ defined 
	in (\ref{SDIPT_BH_w}) is the optimal solution to problem (\textbf{SDIPT$_{BH}$}). \hb
\end{proof}

Based on Theorem \ref{SDIPT_BH_ans}, for any given $T,n,w,u,c,K$,
we can obtain an optimal solution $\bar{w}$ and its corresponding optimal value $\bar{w}(T)$ for
the problem (\textbf{SDIPT$_{BH}$}). For convenience, we denote this subroutine as:
\begin{eqnarray} [ \bar{w},\bar{w}(T) ]=SDIPT_{BH}(T,n,w,u,c,K) \label{SDIPT_BH_alg}\end{eqnarray}

Obviously, this problem only requires one traversal of $O(n)$, so we have the following conclusion.

\begin{theorem}
	The subroutine (\ref{SDIPT_BH_alg}) can solve problem (\textbf{SDIPT$_{BH}$}) in $O(n)$ time.
\end{theorem}

Then we consider the original problem (\textbf{SDIPTC$_{BH}$}). 
Similar to problem (\textbf{SDIPT$_{BH}$}), we can partition $E$ into two parts 
$\tilde E_\leq=  \left \{ e\in E\mid c(e)\leq K \right \} $ and $E \setminus \tilde E_\leq$.
Then, we select elements in $\tilde E_\leq$ according to the value of 
$S(e):= \vert L(e) \vert (u(e)-w(e))$. 
Depending on the relationship between $ \vert \tilde E_\leq \vert $ and the cardinality constraint $N$,
we divide the problem into the following two cases and solve them. 
The first case is rather simple.
\begin{lemma}
	When $ \vert \tilde E_\leq \vert \leq N$,the weight vector $\tilde{w}(e):=
	\begin{cases}
		u(e), & e \in \tilde E_\leq\\
		w(e),&e \in E \setminus \tilde E_\leq\\
	\end{cases}
	$ is an optimal solution to problem (\textbf{SDIPTC$_{BH}$}).
\end{lemma}

When $ \vert \tilde E_\leq \vert >N$, we first prioritize the edges with larger 
$S(e)$ values in the set $\tilde E_\leq$.
We use the $Selection$ algorithm $\tilde{S}_{\leq}^{N}:=Selection(\tilde{S}_{\leq},N)$\cite{ref13} 
to select the $N$-th largest element $\tilde{S}_{\leq}^{N}$ from $\tilde{S}_{\leq}=\left \{S(e) \vert e \in \tilde E_\leq \right \}$.
Based on this, we divide the set $E$ into $\tilde{E}_N:=\left \{e \in \tilde E_\leq \vert  S(e)\ge \tilde{S}_{\leq}^{N} \right \}$ 
and $E \setminus \tilde{E}_N$. 
\begin{theorem} \label{SDIPTC_BH_ans}
	When $ \vert \tilde E_\leq \vert >N$, the weight vector 
	\begin{equation} \label{SDIPTC_BH_w}
		\tilde{w}(e):= 
		\begin{cases}
			u(e), & e \in  \tilde{E}_N\\
			w(e),&e \in E \setminus  \tilde{E}_N
		\end{cases}
	\end{equation}
	is an optimal solution to problem (\textbf{SDIPTC$_{BH}$}).
\end{theorem}
\begin{proof}
	Obviously, $\tilde{w}$ is a feasible solution for problem (\textbf{SDIPTC$_{BH}$}).
	Suppose $\tilde{w}$ is not optimal, while $\hat{w}$ is. Then we have ${\hat{w}}(T)> {\tilde{w}}(T)$ and there exists at least one edge $e_k$ such that $\hat{w}(e_k)> \tilde{w} (e_k)$.
	
	If $e_k\in \tilde{E}_N$, we have $\hat{w}(e_k) > \tilde{w}(e_k)= u(e_k)$, which contradicts the upper bound constraint.
	
	If $e_k\in E\setminus \tilde{E}_N$ and $e_k \notin  \tilde E_\leq$, we have $c(e_k)>K$, which contradicts the updated cost limit.
	
	If $e_k\in E\setminus \tilde{E}_N$ and $e_k \in  \tilde E_\leq$, since $ \vert \tilde E_\leq \vert >N$, there must exist an edge $e_t \in \tilde{E}_N $ such that $\tilde{w}(e_t)>\hat{w}(e_t)$. Moreover, since $S(e)>S(e_k)$ for all $e \in \tilde{E}_N$, we have $S(e_t)>S(e_k)$. In this case, 
	we can construct a new edge weight vector as follows: 
	$$w'(e):= 
	\begin{cases}
		u(e), & e = e_t\\
		w(e),& e =e_k\\
		\hat{w}(e),&\text{otherwise} 
	\end{cases}$$
	and we have $\hat{w}(T)<w'(T)$, which contradicts the assumption. \hb	
\end{proof}

Based on Theorem \ref{SDIPTC_BH_ans}, we propose Algorithm \ref{SDIPTC_BH_alg} for problem (\textbf{SDIPTC$_{BH}$}).
\begin{algorithm}
	\caption{$\left [\tilde{w}, \tilde{w}(T )\right ]=$SDIPTC$_{BH}(T,n,w,u,c,K,N)$}
	\label{SDIPTC_BH_alg}
	\small
	\begin{algorithmic}[1]
		\REQUIRE A rooted tree $T(V,E)$, the number $n$ of edges, the weight vector $w$, an upper bound vector $u$, 
		a cost vector $c$, and given values $K$ and $N$.
		\ENSURE The optimal vector $\tilde{w}$ and the SRD $\tilde{w}(T)$.
		\STATE  $\tilde E_\leq:=\{e\in E\mid c(e)\leq K\}$
		\IF{$ \vert \tilde E_\leq \vert >N$}
		\STATE  $S(e):= \vert L(e) \vert (w'(e)-w(e)),\tilde{S}_{\leq}:=\left \{ S(e)  \vert  e \in \tilde E_\leq \right \}$.
		\STATE  $\tilde{S}_{\leq}^{N}:=Selection(\tilde{S}_{\leq}, N), \tilde{E}_N:=\left \{ e \in \tilde E_\leq   \vert  S(e) \ge \tilde{S}_{\leq}^{N}\right \}.$
		\STATE   Obtain $\tilde{w}$ by   (\ref{SDIPTC_BH_w}).
		\ELSE
		\STATE  $[ \tilde{w},\tilde{w}(T) ]=SDIPT_{BH}(T,n,w,u,c,K)$
		\ENDIF
		\STATE  \textbf{return} $[\tilde{w},\tilde{w}(T)] $.
	\end{algorithmic}
\end{algorithm}

Similar to the proof of Theorem  \ref{thm-complexity-SDIPTC-infty}, we have the following lemma.
\begin{lemma}
    Problem (\textbf{SDIPTC$_{BH}$}) can be solved by Algorithm \ref{SDIPTC_BH_alg} in $O(n)$ time.
\end{lemma}

\subsection{An $O(n\log n)$ time algorithm to solve the problem (\textbf{MCSDIPTC$_{BH}$}) }
Concerning problem (\textbf{MCSDIPTC$_{BH}$}),
similarly, we first consider its sub-problem (\textbf{MCSDIPT$_{BH}$}) 
without cardinality constraint, which can be formulated as follows.

\begin{eqnarray}
	&\min\limits_{\bar{w}}& C(\bar{w}):=\max_{ e\in E}c(e)H({\bar{w}(e),w(e))} \nonumber \\
	\textbf {(MCSDIPT$_{BH}$) } & s.t.&   \sum_{t_k \in Y} {\bar{w}}(P_k) \geq D,\label{MCSDIPT_BH} \\
	&&w(e) \leq \bar{w}(e) \leq u(e), \quad e \in E\nonumber.
\end{eqnarray}

According to the relationship between $u(T)$ and $D$, as well as $w(T)$ and $D$, we discuss the following two cases of the problem separately.
\begin{lemma}
	When $u(T) < D$, problem (\textbf{MCSDIPT$_{BH}$}) is infeasible.
\end{lemma}

\begin{lemma}
If $w(T) \geq D$, then $w$ is an optimal solution to \textbf{(MCSDIPT$_{BH}$)}.
\end{lemma}

When $u(T) \geq D>w(T)$, we first sort the edges in an ascending order by their $c(e)$ 
values and renumber them in the following sequence $\left \{ e_{\alpha_i} \right \} $

\begin{eqnarray}
	c(e_{\alpha_{1}}) \leq c(e_{\alpha_{2}}) \leq \dots \leq c(e_{\alpha_{n}})\label{eq-ce}
\end{eqnarray}

Let $S(e)=\vert L(e) \vert  (u(e)-w(e))$, and we then construct an auxiliary function 
\begin{equation} \label{MCSDIPT_BH_auxg}
    g(k)=\sum_{i=1}^{k} S(e_{\alpha_i})  
\end{equation}
which represents the SDR added when updating the first $k$
edges in the sequence $\left \{ e_{\alpha_i} \right \}$. We use a binary search 
method to iteratively determine $k^*$ such that $g(k^*) < \Delta D \leq g(k^*+1)$ and give the following conclusion.

\begin{theorem}\label{MCSDIPT_BH_ans}
If $u(T) \geq D>w(T)$, then the weight vector $\bar{w}$ defined by:
	\begin{equation}\label{MCSDIPT_BH_w}
		\begin{array}{c}
			\bar{w}(e)=
			\begin{cases}
				u(e),&e \in \bar E_\leq\\
				w(e),&e \in E \backslash \bar E_\leq
			\end{cases}
		\end{array}
	\end{equation}
is an optimal solution to \textbf{(MCSDIPT$_{BH}$)}. Here, $\bar{E}_\leq = \{ e_{\alpha_i} \in E \mid i=1,2,\dots,k^*+1 \}$.
\end{theorem}

\begin{proof}
	It is obvious that $\bar w$ is feasible to the problem (\textbf{MCSDIPT$_{BH}$}).
	Suppose $\bar w$ is not optimal, but $w'$ is. Then $C(w')<C(\bar{w})$.
	Let $e_{\alpha_{j}}$ and $e_{\alpha_{i}}$ be the edges in the sequence $\left \{ e_{\alpha_k} \right \} $ with the maximum cost in $w'$ and $\bar{w}$ respectively.
	Then $c(e_{\alpha_j})=C(w')<C(\bar{w})=c(e_{\alpha_i})$, so $g(j)<g(i)$. Since
	$i$ is the smallest index satisfying $ g(i)\geq \Delta D $, thus $g(j)<\Delta D$, 
	which contradicts the optimality of $w'$. \hb
\end{proof}

Based on Theorem \ref{MCSDIPT_BH_ans}, we present the following Algorithm \ref{MCSDIPT_BH_alg} to solve the problem \textbf{(MCSDIPT$_{BH}$)}.
\begin{algorithm}[!htbp]
	\caption{$\left [\bar{w},C(\bar{w} ) \right ]:=$MCSDIPT$_{BH}(T,n,w,u,c,D)$}
	\label{MCSDIPT_BH_alg}
	\small
	\begin{algorithmic}[1]
		\REQUIRE A rooted tree $T(V,E)$, the number $n$ of edges, the weight vector $w$, an upper bound vector $u$, 
		a cost vector $c$, and a given value $D$.
		\ENSURE The optimal vector $\bar{w}$ and the cost  $C(\bar w)$.
		\STATE  For each $e \in E$, determine the set $L(e)$, calculate 
		$F(e):=c(e)(u(e)-w(e))$, and let $\Delta{D}:=D-w(T)$.
		\STATE  $u(T) := \sum_{t_k \in Y} u(P_k)$.
		\IF{${u}(T)<D$}
		\STATE  \textbf{return} ``The problem is infeasible" and $\bar{w}=[\;], C(\bar{w})= + \infty $.
		\ELSIF{$w(T)>=D$}
		\STATE  \textbf{return} [$w$,0].
		\ELSE
		\STATE  Sort the cost vector $c(e)$ in a non-decreasing order in (\ref{eq-ce}).
		
		\STATE  Set the initial values as $a:=1$, $b:=n$, and $k^*:=0$.
		\WHILE{$b-a>1$ and $k^*=0$}
		\STATE  $k := \left \lfloor \frac{a+b}{2} \right \rfloor$.
		\STATE  Calculate $g(k)$ and $g(k+1)$ using   (\ref{MCSDIPT_BH_auxg}).
		\IF{$g(k) < \Delta D \leq g(k+1)$ }
		\STATE  $k^* := k$.
		\ELSIF{$g(k+1) < \Delta D$}
		\STATE  $a := k$.
		\ELSE
		\STATE  $b := k$.
		\ENDIF
		\ENDWHILE
		
		\STATE  Compute $\bar w$ by  (\ref{MCSDIPT_BH_w}) and $C(\bar w) =c(e_{\alpha_{k^*+1}})$.
		\STATE  \textbf{return} $[\bar{w},C(\bar{w} )] $.
  \ENDIF
	\end{algorithmic}
\end{algorithm}

The following lemma follows from the time complexity of Theorem \ref{thm-complexity-MCSDIPT-infty}.
\begin{lemma}
    Problem (\textbf{MCSDIPT$_{BH}$}) can be solved by Algorithm \ref{MCSDIPT_BH_alg} in $O(n\log{n})$ time.
\end{lemma}

Now we can consider solving the cardinality problem (\textbf{MCSDIPTC$_{BH}$}).
First, we need to determine the infeasibility by Lemma \ref{lem-infea-MCSDIPTC} similarly.
\begin{lemma}
	When $S_N < \Delta D$, problem (\textbf{MCSDIPTC$_{BH}$}) is infeasible.
\end{lemma}

When $S_N \geq \Delta D$, we first run Algorithm \ref{MCSDIPT_BH_alg}:
$$
[w',C(w')] = \text{MCSDIPT}_{BH}(T,n,w,u,c,D)
$$
to obtain the optimal solution to its sub-problem \textbf{(MCSDIPT$_{BH}$)}. This leads to the following lemma.

\begin{lemma}
	
	If $S_N \geq \Delta D$ and $ \sum_{e\in E}H(w'(e),w(e))  \leq N$, then $w'$ is the optimal solution to problem (\textbf{MCSDIPTC$_{BH}$}).
	
\end{lemma}

If $ \sum_{e\in E}H(w'(e),w(e))  > N$, sort the edges $e \in E$ in a non-decreasing order  according to their $c(e)$ values by (\ref{eq-ce}).  For any $N \leq k\leq n$,  let $ E_{k}:=\left \{ e_{\alpha_1}, e_{\alpha_2},\dots, e_{\alpha _k}\right \} $,  $S_{E_k}:=\left \{ S(e) \vert e\in E_k \right \}$,
	call $S_k^{N}:=Selection(S_{E_k},N)$\cite{ref13}.
	Let $E_k^N:=\left \{e \in E_k  \vert S(e)\ge S_k^{N}\right \}$, then
 \begin{equation} h(k)=\sum_{e \in E_k^N } S(e) \label{eq-hk}\end{equation}
is the maximum SDR that can be increased by updating  $N$ edges among the first $k$ edges in the sequence $\left \{ e_{\alpha_i} \right \} $.

Next, we can use binary search for iteration to find the smallest $k^*$ such that $h(k^*)
\geq \Delta D$, and let the updated edge set at this point be $E_{k^*}^{N}$.

\begin{theorem}\label{MCSDIPTC_BH_ans}
	If $ \sum_{e\in E}H(w'(e),w(e))  > N$, the weight vector \begin{equation} \label{MCSDIPTC_BH_ans_w}
		\tilde{w}(e)=
		\begin{cases}
			u(e),& e \in E_{k^*}^{N}\\
			w(e),& e \in E\setminus E_{k^*}^{N}
		\end{cases}
	\end{equation} is an optimal solution to problem (\textbf{MCSDIPTC$_{BH}$}).
\end{theorem}
\begin{proof}
	It is obvious that $\tilde{w}$ is a feasible solution to problem 
	(\textbf{MCSDIPTC$_{BH}$}). 
	Assume that $\tilde{w}$ is not optimal,
	and $\hat{w}$ is an optimal solution with 
	$C(\hat{w}) < C(\tilde{w})$. Suppose that the $N$
	updated edges in $\hat{w}$ correspond to the 
	sequence $\left \{ e_{\alpha_{i}} \right \} $ satisfying $c(e_{\alpha_{p_1}}) \leq c(e_{\alpha_{p_2}}) \leq \cdots \leq c(e_{\alpha_{p_N}})$.
	It is clear that $h(p_N) \geq \Delta D$ and $p_N < k^*$ since $\hat w $ satisfies $c(e_{\alpha_{p_N}})=C(\hat{w}) < C(\tilde{w})=c(e_{\alpha_{k^*}})$. This contradicts the assumption that $k^*$ is the smallest value 
	satisfying $h(k^*) \geq \Delta D$. \hb
\end{proof}

We present the following Algorithm \ref{MCSDIPTC_BH_alg} for solving \textbf{(MCSDIPTC$_{BH}$)} based on Theorem \ref{MCSDIPTC_BH_ans}.

\begin{algorithm}
	\caption{$ \left [\tilde{w},C(\tilde{w} ) \right ]: =$MCSDIPTC$_{BH}(T,n,w,u,c,D,N)$}
	\label{MCSDIPTC_BH_alg}
	\small
	\begin{algorithmic}[1]
		\REQUIRE A rooted tree $T(V,E)$, the number $n$ of edges, the weight vector $w$, an upper bound vector $u$, 
		a cost vector $c$, and two given values $D$ and $N$.
		\ENSURE The optimal vector $\tilde{w}$ and the cost  $C(\tilde{w} )$.
		\STATE  For each $e$, determine the set $L(e)$, calculate $S(e):= \vert L(e) \vert (u(e)-w(e))$, and let $\Delta D:=D-w(T)$.
		\STATE  Sort $S(e)$ in a non-increasing order by (\ref{eq-sort-Se}).
		\STATE  $S_N:=\sum_{i=1}^{N}S({e_{\beta_{i}}})$.
		\IF{$S_N < \Delta D$}
		\STATE  \textbf{return} ``The problem is infeasible" and 
		$\tilde{w}=[\;], C(\tilde{w})= + \infty$.
		\ENDIF

		\STATE  $\left [w',C(w') \right ]:=$MCSDIPT$_{BH}(T,n,w,u,c,D).$
		\IF{$\sum_{e\in E}H(w'(e),w(e)) \leq N$}
		\STATE  \textbf{return} $[w',C(w')]$.
		\ELSE
		\STATE  Sort $c(e)$ in a non-increasing order by (\ref{eq-ce}).
		
		\STATE  $a:=N, b:=n,k^*:=0$
		\WHILE{$a<b \;and \; k^*=0$}
		\STATE  $k := \left \lfloor \frac{a+b}{2} \right \rfloor$.
		\STATE  Compute $E_k^{N}$,  $h(k)$ and $h(k-1)$ by (\ref{eq-hk}).
		\IF{$ h(k-1) <\Delta D \leq h(k) $}
		\STATE  $k^* := k, E_{k^*}^{N}:=E_k^{N}$.
		\ELSIF{$h(k)<\Delta D$}
		\STATE  $a := k$.
		\ELSE
		\STATE  $b := k$.
		\ENDIF
		\ENDWHILE
		\STATE  Compute $\tilde{w}$ by   (\ref{MCSDIPTC_BH_ans_w}) and $C(\tilde{w}) :=c(e_{p_{k^*}})$.
		\STATE  \textbf{return} {$[\tilde{w},C(\tilde{w} )]$.}
		\ENDIF
	\end{algorithmic}
\end{algorithm}

\begin{theorem}
	Problem (\textbf{MCSDIPTC$_{BH}$}) can be solved by Algorithm \ref{MCSDIPTC_BH_alg} in $O(n \log n)$ time.
\end{theorem}

\begin{proof}
	
	In Algorithm \ref{MCSDIPTC_BH_alg}, we sort $S(e)$ and $c(e)$ in Lines 2 and 11, which
	takes $O(n \log n)$ time. 
	Furthermore, in the loop from Line 13 to 23, we use the $Selection(E_k^{N}, N)$ algorithm
	to calculate $h(k)$ 
	in Line 15, which takes $O(n)$ time \cite{ref13}, and the whole loop is searched
	using binary search, so it can be
	completed in $O(n \log n)$ time. Therefore, problem (\textbf{MCSDIPTC$_{BH}$}) 
	can be solved by Algorithm 
	\ref{MCSDIPTC_BH_alg} in $O(n \log n )$ time. \hb
\end{proof}

\section{Numerical Experiments}\label{sec4}

\subsection{An example to show the process of Algorithm \ref{MCSDIPTC_inf_alg} and \ref{MCSDIPTC_BH_alg}}
For the better understanding of Algorithm \ref{MCSDIPTC_inf_alg} and \ref{MCSDIPTC_BH_alg}, Example \ref{example_1} is given to show the detailed computing process.

\begin{example} \label{example_1}

As shown in Fig.\ref{example_1_fig} and \ref{example_2_fig}, let $V:=\left \{ s,v_1,\dots,v_{19} \right \}$, $E:=\left \{  e_1,\dots,e_{19}\right \}$, the corresponding $c,w,L,u$ are shown on edges with different colors. Now we have $w(T):=407$, $u(T):=939$. Suppose the given values are $D:=460$ and $N:=4$.
\end{example}

\subsubsection{Execution of Alg. \ref{MCSDIPTC_inf_alg}}

(1) By defining $S(e):=\vert L(e)\vert(u(e)-w(e))$, the values of $S(e)$ are sorted in a non-increasing order, resulting in $S_N:=296$. As $S_N>\Delta D:=53$, the problem is feasible and has a valid solution.

(2) The values of $\nu(e)$ are calculated and sorted in non-increasing order.

(3) The algorithm is initialized with $flg:=1$, $\tau:=0$, $\bar E^0:=\left \{ e_1,e_3,e_4,e_{13} \right \}$.

(4) The cycle process is presented in the following table \ref{table-Alg3}.

\begin{table}[htbp!]
\centering
\caption{Process of calling Algorithm \ref{MCSDIPTC_inf_alg}.}
\begin{tabular}{|c|c|c|c|}
\hline
$\tau $ & $C(\bar{w}^{\tau})$ & $\bar{E}^{\tau}$     & $\bar{w}(\bar{E}^{\tau}) $  \\ \hline
$0$    & $44.2580$        & $\{e_1,e_3,e_4,e_{13}\}$ & (9.3294,  11.9505,   8.6882,  13.6034) \\ \hline
$0$    & $44.0455$        & $\{e_1,e_2,e_4,e_{13}\}$ & (9.3182, 7.0000,  8.6705,  13.5909) \\ \hline
$0$    & $41.1018$        & $\{e_1,e_2,e_5,e_{13}\}$ &  (9.1633,  7.0000, 20.1102,  13.4178)\\ \hline
$0$    & $33.5776$        & $\{e_1,e_2,e_5,e_{6}\}$  &  (8.7672,  7.0000, 19.3578, 18.7155)\\ \hline
$1$    & $33.5776$        & $\{e_1,e_2,e_5,e_{6}\}$  & (8.7672,  7.0000, 19.3578, 18.7155)\\ \hline
\end{tabular}\label{table-Alg3}
\end{table}
Finally, we get the optimal value $33.5776$. It is worth noting that in the first cycle, we updated the edge set three times with different values $C(\bar w^\tau)$  on Table \ref{table-Alg3}.
Algorithm \ref{MCSDIPT_inf_alg} was called only once at the beginning. We show the upgraded weights of 4 edges $\{e_1,e_2,e_5,e_{6}\}$  on the right figure in Fig. \ref{example_1_fig}.

\begin{figure}[!htbp]
\centering
\includegraphics[width=1.0\linewidth]{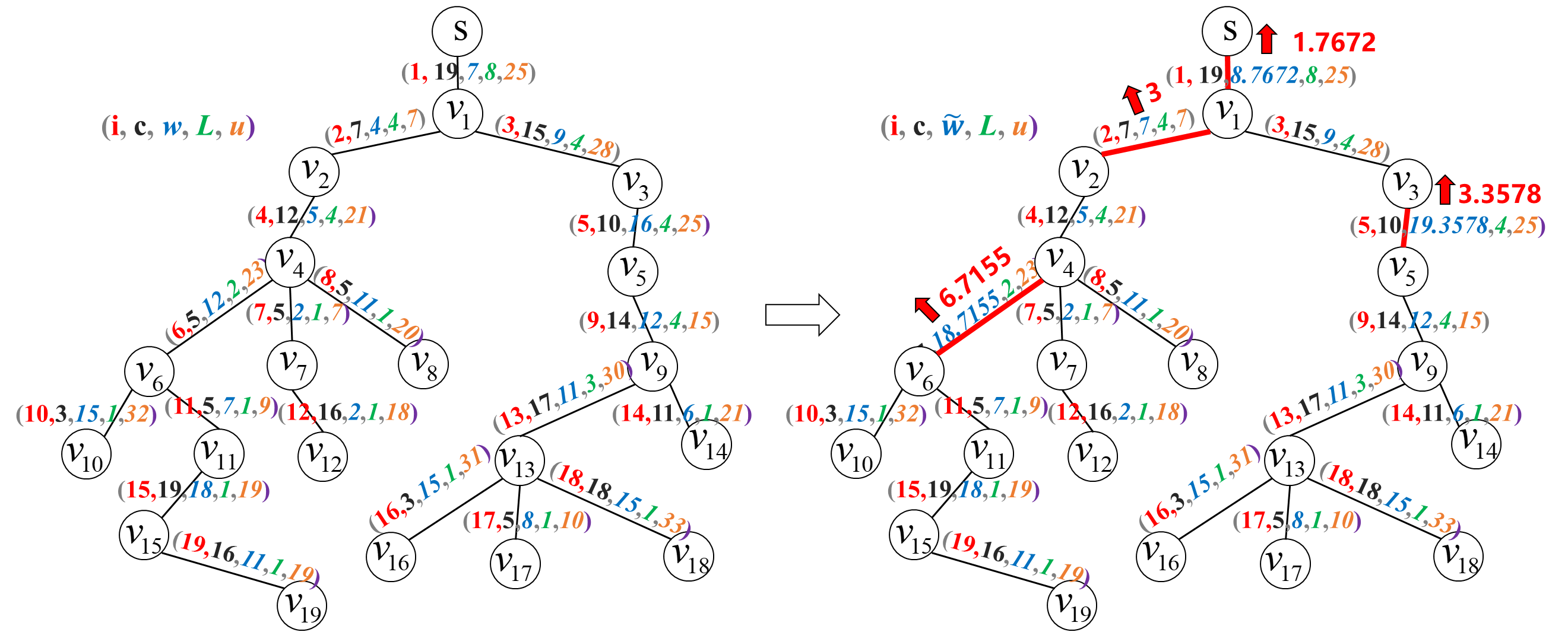}
\caption{The left figure shows the weights $(e_i,c_i,w_i,L_i,u_i)$, where $L_i:=\vert L(e_i) \vert$. The right figure shows the modified edges in Alg. \ref{MCSDIPTC_inf_alg}.}\label{example_1_fig}
\end{figure}

\subsubsection{Execution of Alg. \ref{MCSDIPTC_BH_alg}}

(1) Similarly, sort $S(e)$  in non-increasing order, resulting in $S_N:=296$. As $S_N>\Delta D:=53$, the problem is feasible and has a valid solution.

(2) $\left [w',C(w') \right ]:=$MCSDIPT$_{BH}(T,n,w,u,c,D)$, and we get $k^*:=2$, so 
$
\begin{array}{c}
			\bar{w}(e)=
			\begin{cases}
				u(e),&e \in \left \{e_6, e_{10},e_{16}  \right \} \\
				w(e),&else
			\end{cases}
\end{array}$  and $w'(T):=462>D:=460$.

(3) $\sum_{e \in E} H(w'(e),w(e)):=3 \leq 4$, we get $C(\tilde w)=C(w'):=5$.

\begin{figure}[!htbp]
\centering
\includegraphics[width=1.0\linewidth]{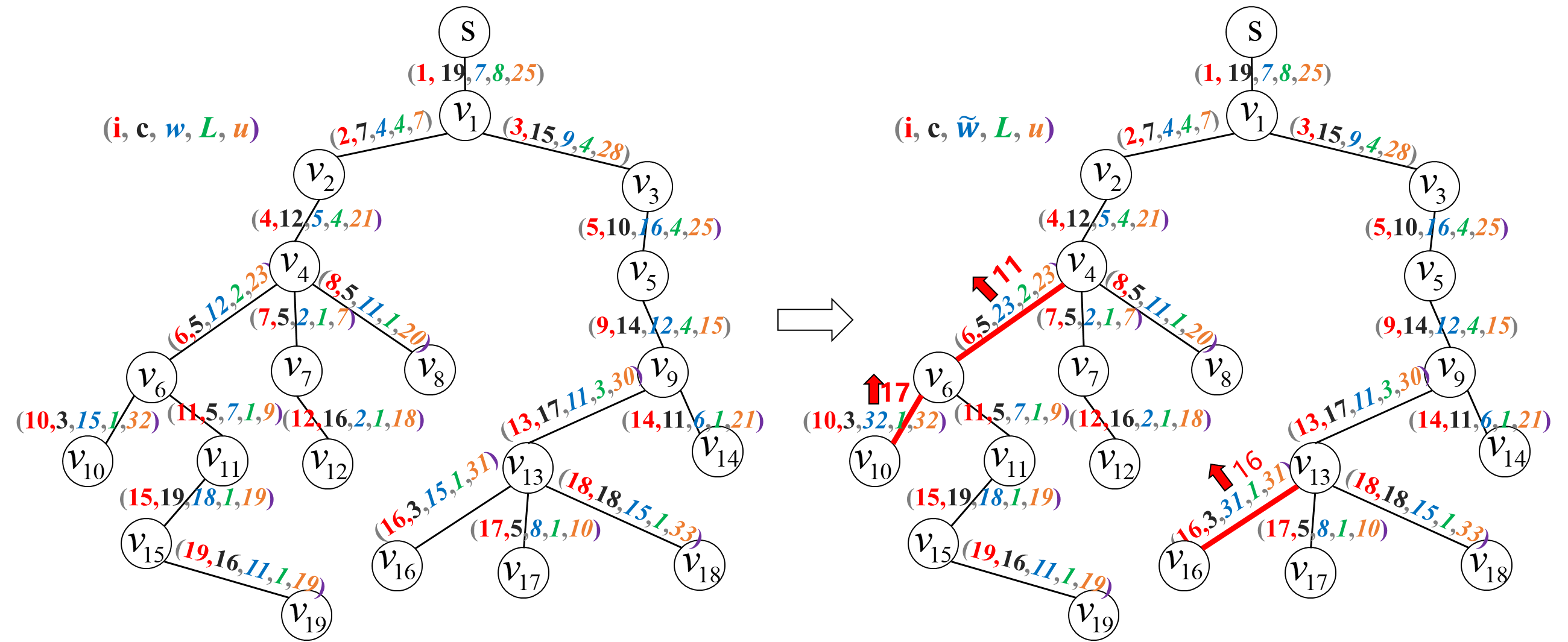}
\caption{The left figure shows the weights $(e_i,c_i,w_i,L_i,u_i)$. The right figure shows the modified edges in Alg. \ref{MCSDIPTC_BH_alg}.}\label{example_2_fig}
\end{figure}

\subsection{Computational experiments}

We present the numerical experimental results for formulas
(\ref{SDIPT_inf_alg}), (\ref{SDIPT_BH_alg}), and algorithms
\ref{SDIPTC_inf_alg}-\ref{MCSDIPTC_BH_alg} in Table \ref{table_1}.
These programs were coded in Matlab2021a and ran on an Intel(R) Core(TM) i7-10875H CPU @ 2.30GHz
and 2.30 GHz machine running Windows 11. We tested these algorithms
on six randomly generated trees with vertex numbers ranging from 1000 to 50000.
We randomly generated the vectors $u$, $c$ and $w$ such that $0\leq w
\leq u$ and $c > 0$.
We randomly generated $K$, $D$  and
$N$ for the problems (\textbf{SDIPT}), (\textbf{MCSDIPT}), (\textbf{SDIPTC}, \textbf{MCSDIPTC}), respectively.
For each randomly generated tree $T$, vectors $u$, $c$, $w$, and values $D$ and $K$, 
we tested formulas (\ref{SDIPT_inf_alg}), (\ref{SDIPT_BH_alg}), and algorithms
\ref{SDIPTC_inf_alg}-\ref{MCSDIPTC_BH_alg}, whose average, maximum and minimum CPU time are denoted by $T_i$, $T_i^{max}$ and $T_i^{min}$, respectively, where $i = {1, \cdots, 8}$. 

From Table \ref{table_1}, we can see that Algorithm \ref{MCSDIPTC_inf_alg} 
is the most time-consuming due to the $O(Nn^2)$ complexity of its program and the uncertainty of its iteration number.
Formulas (\ref{SDIPT_inf_alg}) and (\ref{SDIPT_BH_alg}) are relatively simple and require only one traversal of the vector, so they are less time-consuming.

Overall, these algorithms are all very effective and follow their respective time complexities well. When $n$ is small, the time differences among the three algorithms are relatively small, but as $n$ increases, the differences between the algorithms become more pronounced.

\begin{table}[!htbp] 
	\centering
	\caption{Performance of Algorithms}
	\label{table_1} 
	\begin{tabular}{ccccccc}\hline 
		
		\renewcommand{\arraystretch}{1.5pt}
		\renewcommand{\tabcolsep}{5pt}
		Complexity& $n$ &1000 &5000 &10000 &30000    &50000 \\ \toprule
		
		$O(n)$ & $T_1$        &0.0002 & 0.0003 &0.0005 &0.0013 &0.0019  \\
		&  $T_1^{max}$ & 0.0005 &0.0021 &0.0038 &0.0107 &0.0186 \\
		&  $T_1^{min}$ & 0.0001 &0.0001 & 0.0002 &0.0007 &0.0014 \\
		
		$O(n)$ & $T_2$ &0.0008 &0.0037 &0.0084 &0.0251 &0.0393 \\
		& $T_2^{max}$ &0.0022 &0.0079 &0.0095 &0.0274 &0.0428 \\
		& $T_2^{min}$ &0.0003 &0.0017 &0.0029 &0.0160 &0.0251 \\		
		
		$O(n \log n)$ & $T_3$ &0.0013 &0.0062 &0.0140 &0.0679 &0.1477 \\
		& $T_3^{max}$ &0.0075 &0.0122 &0.0217 &0.0805 &0.1802 \\
		& $T_3^{min}$ &0.0010 &0.0036 &0.0109 &0.0561 &0.1365 \\

		$O(Nn^2)$ & $T_4$ &0.0004 &0.0612&0.3575 &11.24 &47.66  \\
		& $T_4^{max}$ &0.0013 &0.0682 &0.4720 &24.36 &62.34 \\
		& $T_4^{min}$ &0.0002 &0.0446 &0.2654 &6.331 &40.98 \\

		$O(n)$ & $T_5$ &0.0007 &0.0037 &0.0081 &0.0246 &0.0403 \\
		& $T_5^{max}$ &0.0023 &0.0071 &0.0100 &0.0267 &0.0429 \\
		& $T_5^{min}$ &0.0004 &0.0031 &0.0069 &0.0221 &0.0387 \\
		
		$O(n)$ & $T_6$ &0.0012 &0.0057 &0.0114 &0.0320 &0.0715 \\
		& $T_6^{max}$ &0.0063 &0.0098 &0.0135 &0.0493 &0.0974 \\
		& $T_6^{min}$ &0.0007 &0.0042 &0.0096 &0.0262 &0.0638 \\		
		
		$O(n\log n)$ & $T_7$ &0.0016 &0.0067 &0.0150 &0.0683 &0.1578 \\
		& $T_7^{max}$ &0.0064 &0.0119 &0.0413 &0.0893 &0.1917 \\
		& $T_7^{min}$ &0.0008 &0.0052 &0.0120 &0.0602 &0.1424 \\
		
		$O(n\log n)$ & $T_8$ &0.0021 &0.0097 &0.0209 &0.0781 &0.1792 \\
		& $T_8^{max}$ &0.0077 &0.0141 &0.0425 &0.0914 &0.2024 \\
		& $T_8^{min}$ &0.0012 &0.0085 &0.0187 &0.0774 &0.1623 \\
		\hline
	\end{tabular}
\end{table}

\section{Conclusion and further research}\label{sec5}
This paper proposes two linear time greedy algorithms for problems (\textbf{SDIPT$_ \infty$}) under weighted
    $ l_\infty$ norm and the problem (\textbf{SDIPT$_{BH}$}) under bottleneck Hamming distance, respectively. 
     For their relevant minimum cost problems (\textbf{MCSDIPT$_ \infty$})  and (\textbf{MCSDIPT$_{BH}$}), 
    two $O(n \log n)$ time algorithms  are proposed based on  binary methods, respectively.
     In addition, this paper considers the sum of root-leaf distance interdiction problem with cardinality constraints  by upgrading edges on trees(\textbf{SDIPTC}) and its relevant minimum cost problem (\textbf{MCSDIPTC}). 
    Then two binary search algorithms are proposed within $O(n \log n)$ time for the problems (\textbf{SDIPTC})  
    under weighted $l_\infty$ norm and  weighted bottleneck Hamming distance, respectively. 
    For problems (\textbf{MCSDIPTC}), two binary search algorithms within $O(Nn^2)$ and $O(n \log n)$ under weighted $l_\infty$ norm and  weighted bottleneck Hamming distance are proposed,  respectively.

We validate the effectiveness of the proposed algorithms through numerical experiments. As future work, we aim to extend our approach to interdiction problems on source-sink path length, maximum flow, and minimum spanning tree under different measurements for general graphs.

\vskip 0.5cm
{\small \textbf{Funding} The Funding was provided by National Natural Science Foundation of China (grant no: 11471073).

\vskip 0.3cm
\textbf{Data availability}
 Data sharing is not applicable to this article as our datasets were generated randomly.

\section*{Declarations}

\textbf{Competing interests} The authors declare that they have no competing interest.

}

%
%



\end{document}